\documentclass[10pt]{amsart}

\usepackage[utf8]{inputenc}
\usepackage{amsmath,amssymb,amsthm}
\usepackage{mathpazo}
\usepackage[a4paper,margin=1in]{geometry}
\usepackage{hyperref}
\hypersetup{hidelinks}

\newtheorem{theorem}{Theorem}[section]
\newtheorem{lemma}[theorem]{Lemma}
\newtheorem{proposition}[theorem]{Proposition}

\numberwithin{equation}{section}

\newcommand{\rpc}{\operatorname{rpc}}
\newcommand{\col}{\operatorname{col}}
\newcommand{\E}{\mathbb E}
\newcommand{\Pp}{\mathbb P}
\newcommand{\cH}{\mathcal H}
\newcommand{\cP}{\mathcal P}

\title[Sharp rainbow path covers]
{Sharp Rainbow Path Covers in Dense and Complete Multipartite Graphs}

\author{Xiao-Chuan Liu}
\address[Liu]{Departamento de Matemática,
 Universidade Federal de Pernambuco,
	Avenida Jornalista Aníbal Fernandes, Cidade Universitária, Recife, Brasil}
\email{xiaochuan.liu@ufpe.br}
\author{Boyan Xu}
\address[Xu]{School of Data Science and Information Engineering, Guizhou Minzu University, Guiyang, Guizhou Province, 550025, China}
\email{boyan04518@gmail.com}
\author{Xu Yang}
\address[Yang]{Instituto de Computação, Universidade Federal de Alagoas,
	Av. Lourival Melo Mota, S/N, Maceió, 57072-900, Brasil}
\email{yang@ic.ufal.br}

\begin{document}

\begin{abstract}
A path in a properly edge-colored graph is rainbow if its edges have pairwise distinct colors. For a proper edge-coloring $c$ of a graph $G$, let $\rpc(G,c)$ be the minimum number of rainbow paths needed to cover $E(G)$, and let $\rpc(G)$ be the maximum of $\rpc(G,c)$ over all proper edge-colorings of $G$. We prove that, for every fixed $0<\alpha<1$, every properly edge-colored $n$-vertex graph with minimum degree at least $\alpha n$ satisfies $\rpc(G,c)\leq(1+o(1))n/2$, where the coefficient $1/2$ is best possible. We also determine $\rpc(G)$ asymptotically for every complete multipartite graph. If $G=K_{n_1,\ldots,n_r}$ has order $n$ and largest and smallest part sizes $M$ and $s$, respectively, then, uniformly over all choices of the number and sizes of the parts, $\rpc(G)=(1+o(1))\max\{\min\{\lfloor n/2\rfloor,n-M\},(n-s)/2\}$. The proof combines pseudorandom packings of globally rainbow linear forests with a decomposition into dense parts and prescribed avoidance for arbitrary dense graphs, and with reserved connectors and a direct dominant-part argument for complete multipartite graphs.
\end{abstract}

\maketitle

\section{Introduction}\label{sec:introduction}

Throughout the paper, all graphs are finite and simple. Let $G$ be equipped
with a proper edge-coloring $c$, so adjacent edges receive different
colors. A path is \textbf{rainbow} if its edges have pairwise distinct
colors. Write $\rpc(G,c)$ for the minimum number of rainbow simple paths
whose union contains $E(G)$, and put $\rpc(G):=\max_c\rpc(G,c)$, where
the maximum is taken over all proper edge-colorings of $G$. Paths in a
cover need not be edge-disjoint, and different paths may share vertices,
edges, and colors.

Path covering without the rainbow requirement goes back to Chung, who
conjectured that the edges of every connected $n$-vertex graph can be
covered by at most $\lceil n/2\rceil$ paths~\cite{Chung}. As in our
setting, the paths in such a cover need not be edge-disjoint. Pyber
proved an asymptotic version of Chung's conjecture~\cite{Pyber}, and Fan
subsequently proved the conjecture exactly~\cite{Fan}. Bonamy, Botler,
Dross, Naia, and Skokan asked whether a linear bound persists when every
path is required to be rainbow: does there exist an absolute constant
$C$ such that the edges of every properly edge-colored $n$-vertex graph
can be covered by at most $Cn$ rainbow paths?~\cite{BBDNS} A greedy
argument gives an $O(n\log n)$ bound for arbitrary graphs, but the
conjectured $O(n)$ bound remains open. It has, however, been established
for several models of random graphs~\cite{KMN,Fernandes,BKMMS}.

Dense graphs are a natural setting in which to seek a sharp form of this
conjecture. They were proposed in Problem~9 of the BIRS workshop report on
rainbow path coverings~\cite{Granet}; the GAPCOMB problem booklet also
records the general question as Conjecture~19 and, in Problem~20,
explicitly suggests graphs of large minimum degree as a class to
investigate~\cite{Naia}. The minimum-degree condition also has a natural
precedent in classical decomposition theory. Gir\~ao, Granet, K\"uhn, and
Osthus proved that every sufficiently large $n$-vertex graph of linear
minimum degree can be decomposed into at most $n/2+o(n)$
paths~\cite{GiraoGranetKuhnOsthus}. Their paths need not be rainbow and,
unlike the paths considered here, form an edge decomposition. However,
their result identifies $n/2$ as the natural first-order benchmark for
dense graphs. Our first theorem attains the same first-order bound for
rainbow path covers whenever the minimum degree is a fixed positive
fraction of the order. The density hypothesis is essential for this
sharper estimate: the construction in \eqref{eq:k6-general-lower} shows
that arbitrary properly edge-colored graphs may require $5n/6-O(1)$
rainbow paths. Thus a universal $n/2+o(n)$ rainbow-cover bound is false,
even though the general $O(n)$ conjecture remains open. In the dense
setting, the two different problems---uncolored path decomposition and
rainbow path covering---nevertheless have the same first-order bound.

The search for individual long rainbow paths in properly colored
complete graphs has a separate history. Andersen conjectured that every
properly colored $K_n$ contains a rainbow path on $n-1$
vertices~\cite{Andersen}. Alon, Pokrovskiy, and Sudakov, and later Balogh
and Molla, obtained successively stronger almost-spanning
versions~\cite{AlonPokrovskiySudakov,BaloghMolla}. Bowtell, Montgomery,
M\"uyesser, and Pokrovskiy recently proved Andersen's conjecture for all
sufficiently large $n$~\cite{BowtellMontgomeryMuyesserPokrovskiy}. These
results concern the existence of a single long rainbow path. A path cover
is a global object: even for $K_n$, it must account for all
$\Theta(n^2)$ edges, and one long rainbow path does not by itself provide
a way to cover the remaining edges.

Another related direction seeks many edge-disjoint rainbow spanning
structures. Pokrovskiy and Sudakov found linearly many rainbow spanning
trees in every properly colored complete graph~\cite{PokrovskiySudakov},
and Montgomery, Pokrovskiy, and Sudakov later obtained an asymptotic
decomposition into rainbow spanning
trees~\cite{MontgomeryPokrovskiySudakov}. Paths are considerably more
rigid than trees, since they have maximum degree two. Moreover, our aim
is to cover every edge of an arbitrary dense graph. The following theorem
gives the asymptotically best possible universal bound in this setting.

\begin{theorem}\label{thm:dense-main}
For every $0<\alpha<1$ and every $\varepsilon>0$, there exists
$n_0=n_0(\alpha,\varepsilon)$ such that the following holds. If $G$ is
a graph on $n\geq n_0$ vertices with $\delta(G)\geq\alpha n$, then
every proper edge-coloring $c$ of $G$ satisfies
\begin{equation}\label{eq:dense-main-upper}
 \rpc(G,c)\leq(1+\varepsilon)\frac n2.
\end{equation}
\end{theorem}

The coefficient $1/2$ is necessary already for complete graphs. More
generally, two elementary obstructions give the natural scale for the
multipartite problem. Let $\nu(G)$ be the matching number of $G$, and let
$\Delta(G)$ be its maximum degree. A simple path contains at most two
edges incident with a fixed vertex, so every path cover has at least
$\lceil\Delta(G)/2\rceil$ paths. For the second obstruction, let $N$ be
a maximum matching. Give all edges of $N$ one color and give every edge
outside $N$ its own new color. This coloring is proper because the edges
of $N$ are pairwise disjoint. Every rainbow path contains at most one edge
of $N$, so this coloring requires at least $|N|=\nu(G)$ paths. Hence
\begin{equation}\label{eq:universal-lower}
 \rpc(G)\geq\max\{\nu(G),\lceil\Delta(G)/2\rceil\}.
\end{equation}
Our second result shows that these two obstructions are asymptotically
sharp for every complete multipartite graph. Let
\begin{equation}
 G=K_{n_1,\ldots,n_r},\qquad
 n=\sum_{i=1}^r n_i,\qquad
 M=\max_i n_i,\qquad s=\min_i n_i.
\end{equation}
A vertex in a part of size $t$ has degree $n-t$, so
$\Delta(G)=n-s$. Moreover,
$\nu(G)=\min\{\lfloor n/2\rfloor,n-M\}$. Indeed, both terms are
upper bounds. If $M>n/2$, match every vertex outside a largest part to
a distinct vertex in that part. If $M\leq n/2$, repeatedly match
vertices from two largest non-empty parts. After each step, either at
most one vertex remains or no part contains more than half of the
remaining vertices. Hence the process produces a matching of size
$\lfloor n/2\rfloor$.

\begin{theorem}\label{thm:multipartite-main}
For every $\varepsilon>0$, there exists $n_0=n_0(\varepsilon)$ such
that the following holds. Let $G=K_{n_1,\ldots,n_r}$ have at least two
non-empty parts and order $n\geq n_0$. Then every proper edge-coloring
$c$ of $G$ satisfies
\begin{equation}\label{eq:multipartite-upper}
 \rpc(G,c)\leq(1+\varepsilon)
 \max\{\nu(G),\Delta(G)/2\}.
\end{equation}
Consequently, uniformly over all complete multipartite graphs of order
$n$,
\begin{equation}\label{eq:multipartite-asymptotic}
 \rpc(G)=(1+o(1))
 \max\{\min\{\lfloor n/2\rfloor,n-M\},(n-s)/2\}.
\end{equation}
\end{theorem}

For example, Theorems~\ref{thm:dense-main} and~\ref{thm:multipartite-main}
both give $\rpc(K_n)=(1/2+o(1))n$, while
Theorem~\ref{thm:multipartite-main} gives
$\rpc(K_{n,n})=(1+o(1))n$.

For general graphs, larger lower bounds arise from small components.
A disjoint union of properly colored copies of $K_3$ requires two
rainbow paths per component and therefore gives $2n/3$ paths. There is
also a proper edge-coloring $c$ of $K_6$, obtained from a
one-factorization, for which four rainbow paths do not cover all the
edges but five do, so $\rpc(K_6,c)=5$. Disjoint unions of copies of this
colored $K_6$, together with isolated vertices, therefore give
\begin{equation}\label{eq:k6-general-lower}
 \rpc(G,c)\geq5\left\lfloor\frac n6\right\rfloor
 =\frac56n-O(1).
\end{equation}
No larger asymptotic coefficient is known to us.

The multipartite result is uniform even when the minimum degree is
sublinear, so it is not a consequence of Theorem~\ref{thm:dense-main}.
It therefore reaches well beyond the dense regime. It shows that the
matching-number and maximum-degree obstructions are the only asymptotic
obstructions throughout the full class of complete multipartite graphs,
independently of the number and relative sizes of the parts.

Both upper bounds use the same path-packing principle. A stationary
non-backtracking walk gives a weighted family of short rainbow paths
with almost uniform edge marginals and small pairwise marginals. The
weighted matching theorem of Ehard, Glock, and Joos then packs most
edges into globally rainbow linear forests. For arbitrary dense graphs,
we first decompose the graph into a bounded number of dense parts and
organize the sparse interface into ordered lists of pieces. A packing
lemma with prescribed forbidden sets chooses the internal forests while
keeping private vertex and color reserves disjoint from both the forests
and the prescribed pieces. These reserves are then used to join each list
to its forest. For complete multipartite graphs, the part structure gives
shorter connectors and the sharper bound in terms of the matching number
and maximum degree.

Thus, the proof for arbitrary dense graphs may be viewed as a coarse
regularity-style argument, with carefully chosen reserves of both
vertices and colors.

The framework may be useful in other rainbow decomposition problems.
The stationary walk supplies a pseudorandom weighted family of short
rainbow paths in an arbitrary dense graph, and the indexed hypergraph
matching converts it into many globally rainbow linear forests while
allowing different forbidden vertices and colors for different forests.
Together with the decomposition into robust parts, this permits the
exceptional edges to be prescribed before the main packing is chosen.
This separation between the exceptional edges and the bulk packing is
suited to problems in which many rainbow structures must be constructed
simultaneously, rather than one long rainbow structure.

Section~\ref{sec:preliminaries} records the concentration inequalities,
the stationary-walk estimates, and the Ehard--Glock--Joos theorem. The
main result of
Section~\ref{sec:linear-minimum-degree} is the two-part
Theorem~\ref{thm:linear-minimum-degree}. Part~(i) is exactly
Theorem~\ref{thm:dense-main}, while part~(ii) proves the stronger
complete multipartite estimate under the same linear minimum-degree
hypothesis. Subsection~\ref{subsec:forest-packing} gives the basic
rainbow linear-forest packing mechanism. Subsection~\ref{subsec:dense-multipartite}
gives the multipartite completion and proves part~(ii).
Subsection~\ref{subsec:robust-tools} develops the packing with forbidden
sets. Subsection~\ref{subsec:arbitrary-dense} then gives the decomposition
into dense parts and proves part~(i). Finally,
Section~\ref{sec:complete-multipartite} treats the dominant-part case
and completes the proof of Theorem~\ref{thm:multipartite-main}.

\section{Preliminaries}\label{sec:preliminaries}

This section collects the probabilistic tools used throughout the paper.
We first record the standard concentration inequalities used in the random
constructions; we refer to~\cite{AlonSpencer} for these inequalities and
their usual variants. We then establish estimates for a stationary
non-backtracking walk in a properly edge-colored dense graph and state the
weighted pseudorandom matching theorem of Ehard, Glock, and Joos.

For a subgraph or walk $F$ of $G$, let $\col(F)$ be the set of colors
on its edges, and abbreviate $d_G(v)$ to $d(v)$. For a color $\gamma$,
write $E_\gamma=\{e\in E(G):c(e)=\gamma\}$, and let
$V(E_\gamma)$ be the set of vertices incident with an edge of
$E_\gamma$. For a subgraph $F\subseteq G$, put
$\mu_c(F)=\max_\gamma|E_\gamma\cap E(F)|$; in particular,
$\mu_c(G)=\max_\gamma|E_\gamma|$. Every color class is a matching, so
$\mu_c(G)\leq\nu(G)$.

For a hypergraph $\cH$, let $\Delta(\cH)$ be its maximum degree and let
\begin{equation}
 \Delta_2(\cH):=\max_{x\neq y}
 |\{f\in E(\cH):x,y\in f\}|
\end{equation}
be its maximum codegree.

\subsection{Concentration inequalities}\label{subsec:concentration}

The following standard inequalities, as well as the variants used below,
can be found in~\cite{AlonSpencer}. Throughout this subsection,
$c_{\rm prob}>0$ denotes an absolute constant. Chernoff's inequality
gives the following two bounds.
If $X$ is a sum of independent Bernoulli variables
with mean $\mu$, then, for $0<\rho\leq1$,
\begin{equation}\label{eq:chernoff-standard}
 \Pp(|X-\mu|>\rho\mu)
 \leq 2\exp(-c_{\rm prob}\rho^2\mu).
\end{equation}
We also use the general upper-tail form
\begin{equation}\label{eq:chernoff-upper}
 \Pp(X\geq\mu+s)
 \leq \exp\left(-c_{\rm prob}\min\left\{\frac{s^2}{\mu},s\right\}\right)
 \qquad(s>0).
\end{equation}
If $X_1,\ldots,X_N$ are independent and $a_i\leq X_i\leq b_i$, then
Hoeffding's inequality gives, for every $s>0$,
\begin{equation}\label{eq:hoeffding-standard}
 \Pp\left(\left|\sum_{i=1}^N X_i-
 \E\sum_{i=1}^N X_i\right|>s\right)
 \leq 2\exp\left(-\frac{2s^2}
 {\sum_{i=1}^N(b_i-a_i)^2}\right).
\end{equation}
If $Z$ is a function of independent variables and changing the $i$th
variable changes $Z$ by at most $c_i$, then the bounded-differences
inequality gives
\begin{equation}\label{eq:mcdiarmid-standard}
 \Pp(|Z-\E Z|>s)
 \leq 2\exp\left(-\frac{2s^2}{\sum_i c_i^2}\right).
\end{equation}

We also use the following lower-tail form of Janson's inequality. Let
$\{\xi_u:u\in\Omega\}$ be independent Bernoulli variables, and let
$(A_i)_{i\in I}$ be a finite indexed family of subsets of $\Omega$;
repetitions are allowed. Put
$I_i=\prod_{u\in A_i}\xi_u$, $X=\sum_{i\in I}I_i$, and
$\mu=\E X$. Write $i\sim j$ if $i\neq j$ and
$A_i\cap A_j\neq\varnothing$, and put
\begin{equation}\label{eq:janson-dependency-sum}
 \Delta=\sum_{\{i,j\}:i\sim j}\E(I_iI_j).
\end{equation}
Then, for $0\leq s\leq\mu$,
\begin{equation}\label{eq:janson-standard}
 \Pp(X\leq\mu-s)
 \leq \exp\left(-\frac{c_{\rm prob}s^2}{\mu+\Delta}\right).
\end{equation}

We shall use these inequalities only at simple scales. If $N=O(n)$,
the variables $X_1,\ldots,X_N$ are independent,
$0\leq X_i\leq Cn^{-2}$, and $\E\sum_iX_i=\Theta(n^{-1})$, then
\eqref{eq:hoeffding-standard}, with
$s=\rho\E\sum_iX_i$, gives failure probability
$\exp(-\Omega(\rho^2n))$. A Bernoulli sum of mean $\mu$ has relative
error at most $\rho$ with failure probability
$\exp(-\Omega(\rho^2\mu))$ by \eqref{eq:chernoff-standard}. In each
application of \eqref{eq:mcdiarmid-standard}, we compute
$\sum_i c_i^2$ explicitly. Finally, if $h\geq2$ is fixed,
$\mu=\Theta(n^{h-1})$ and $\Delta=O(n^{2h-3})$, then
\eqref{eq:janson-standard}, with $s=\mu/2$, gives failure probability
$\exp(-\Omega(n))$.

\subsection{Stationary rainbow paths}\label{subsec:stationary}
Fix $\alpha>0$ and an integer $\ell\geq3$, both independent of $n$.
All asymptotic estimates in this subsection are taken as $n\to\infty$
with $\alpha$ and $\ell$ fixed.

Let $(G,c)$ be a properly
edge-colored graph on $n$ vertices with $\delta(G)\geq\alpha n$, and
put $m=e(G)$. Write
$\vec E(G)=\{(u,v):uv\in E(G)\}$. Thus each edge $uv$ gives the two
possible moves $(u,v)$ and $(v,u)$. Choose
$(X_0,X_1)$ uniformly from $\vec E(G)$. For
$j=1,\ldots,\ell-1$, having reached $X_j$ from $X_{j-1}$, choose
$X_{j+1}$ uniformly from $N(X_j)\setminus\{X_{j-1}\}$. This defines
the \textbf{non-backtracking walk} $\mathbf X=(X_0,\ldots,X_\ell)$.

Put $Y_j=(X_{j-1},X_j)$ for $j=1,\ldots,\ell$. A distribution $\pi$
on $\vec E(G)$ is \textbf{stationary} if $Y_j\sim\pi$ implies
$Y_{j+1}\sim\pi$. Thus, if $Y_1$ has a stationary distribution, then
every $Y_j$ has the same distribution. The next lemma shows that the
uniform distribution on $\vec E(G)$ is stationary and records the
estimates that we need.

\begin{lemma}[Stationary path estimates]\label{lem:stationary}
For every $j=1,\ldots,\ell$ and every $(u,v)\in\vec E(G)$,
\begin{equation}\label{eq:stationary-directed-edge}
 \Pp(Y_j=(u,v))=\frac1{2m}.
\end{equation}
In particular, the edge traversed at every position is uniform in $E(G)$,
and $\Pp(X_j=v)=d(v)/(2m)$ for every $v\in V(G)$ and
$j=0,\ldots,\ell$.

The probability that $\mathbf X$ is not a simple rainbow path is $O(n^{-1})$.
Let $\cP$ be the set of unoriented simple rainbow $\ell$-edge paths,
and assign to each $P\in\cP$ the weight
\begin{equation}
 w(P):=\Pp(\mathbf X\text{ traverses }P\text{ in either direction}).
\end{equation}
For $P\in\cP$ and $v\in V(G)$, let $d_P(v)$ denote the number of edges
of $P$ incident with $v$. Thus $d_P(v)=0$ if $v\notin V(P)$,
$d_P(v)=1$ if $v$ is an endpoint of $P$, and $d_P(v)=2$ if $v$ is an
internal vertex of $P$. Uniformly over
$e\in E(G)$, $v\in V(G)$, and all colors $\gamma$,
\begin{align}
 \sum_{\substack{P\in\cP\\e\in E(P)}}w(P)
  &=(1+o(1))\frac{\ell}{m},\label{eq:edge-marginal}\\
 \sum_{\substack{P\in\cP\\v\in V(P)}}w(P)
  &=(1+o(1))\frac{(\ell+1)d(v)}{2m},
  \label{eq:vertex-marginal}\\
 \sum_{P\in\cP}d_P(v)w(P)
  &=(1+o(1))\frac{\ell d(v)}{m},\label{eq:incidence-marginal}\\
 \sum_{\substack{P\in\cP\\\gamma\in\col(P)}}w(P)
  &=(1+o(1))\frac{\ell|E_\gamma|}{m}.
  \label{eq:color-marginal}
\end{align}
Moreover, there are constants $c_0=c_0(\alpha,\ell)>0$ and
$C=C(\alpha,\ell)>0$ such that
\begin{equation}\label{eq:path-weight-bound}
 c_0n^{-\ell-1}\leq w(P)\leq Cn^{-\ell-1}
\end{equation}
for every $P\in\cP$. Thus
$w(P)=\Theta_{\alpha,\ell}(n^{-\ell-1})$ uniformly over $P\in\cP$.
If $u\neq v$ and $\gamma\neq\gamma'$, then
\begin{align}
 \Pp(u,v\in V(\mathbf X))
  &=O(n^{-2}),\label{eq:two-vertices}\\
 \Pp(v\in V(\mathbf X),\gamma\in\col(\mathbf X))
  &=O\biggl(\frac{\mathbf1_{\{v\in V(E_\gamma)\}}}{n^2}
       +\frac{|E_\gamma|}{n^3}\biggr),
       \label{eq:vertex-color}\\
 \Pp(\gamma,\gamma'\in\col(\mathbf X))
  &=O\biggl(\frac{\min\{|E_\gamma|,|E_{\gamma'}|\}}{n^3}\biggr).
  \label{eq:two-colors}
\end{align}
All error terms are uniform, with constants depending only on
$\alpha$ and $\ell$.
\end{lemma}

\begin{proof}
Fix $(v,w)\in\vec E(G)$. Its possible predecessors are the directed
edges $(u,v)$ with $u\in N(v)\setminus\{w\}$. If $Y_j$ is uniform on
$\vec E(G)$, then
\begin{equation}
 \Pp(Y_{j+1}=(v,w))
 =\sum_{u\in N(v)\setminus\{w\}}
   \frac1{2m}\frac1{d(v)-1}
 =\frac1{2m}.
\end{equation}
Thus the uniform distribution is stationary. Since $Y_1$ is uniform,
induction proves \eqref{eq:stationary-directed-edge}. Taking the
underlying edge, head, and tail marginals shows that every edge
traversed by $\mathbf X$ is uniform in $E(G)$ and that
$\Pp(X_j=v)=d(v)/(2m)$ for every $j$.

Let $\mathcal A$ be the event that $\mathbf X$ is a simple rainbow path.
At each step, the next vertex is chosen from at least $\alpha n-1$
vertices. There are at most $\ell$ previously visited vertices to
avoid. There are also at most $\ell$ previously used colors, and
properness implies that at most one edge of each such color is
incident with the current vertex. Hence the conditional probability
of repeating a vertex or a color at any step is $O(n^{-1})$. Since
$\ell$ is fixed, a union bound gives
$\Pp(\mathcal A^c)=O(n^{-1})$.

The reversed chain has the same form. Conditioned on
$Y_j=(v,w)$, its predecessor is uniform among the directed edges
$(u,v)$ with $u\in N(v)\setminus\{w\}$. We may therefore expose the
walk in both directions from any fixed edge or vertex position. The
same union-bound argument then shows that
$\Pp(\mathcal A^c\mid\mathcal B)=O(n^{-1})$ whenever $\mathcal B$ is a
positive-probability event specifying an edge, a vertex, or a color at
a prescribed position.

On $\mathcal A$, every edge, vertex, and color occurs at most once.
At each of the $\ell$ edge positions, a fixed edge occurs with
probability $1/m$, while a fixed color $\gamma$ occurs with
probability $|E_\gamma|/m$. At each of the $\ell+1$ vertex positions,
$v$ occurs with probability $d(v)/(2m)$. Finally, at each edge
position the probability that the traversed edge is incident with
$v$ is $d(v)/m$. Summing over the relevant positions and using the
conditional estimate above proves
\eqref{eq:edge-marginal}--\eqref{eq:color-marginal}.

Fix $P=v_0\cdots v_\ell\in\cP$. The probability that $\mathbf X$ follows
this orientation is
\begin{equation}
 \frac1{2m}\prod_{j=1}^{\ell-1}\frac1{d(v_j)-1}.
\end{equation}
The reverse orientation has the same probability, since it has the
same internal vertices. Hence
\begin{equation}\label{eq:exact-path-weight}
 w(P)=\frac1m\prod_{j=1}^{\ell-1}\frac1{d(v_j)-1}.
\end{equation}
For all sufficiently large $n$, we have
$\alpha n^2/2\leq m\leq n^2/2$ and
$\alpha n/2\leq d(v_j)-1\leq n$. Substitution in
\eqref{eq:exact-path-weight} proves \eqref{eq:path-weight-bound}.

It remains to prove the joint estimates. We first record the
conditioning used below. Fix a directed edge $\vec e$ and two edge
positions $r<s$. If $s=r+1$, properness gives at most one permissible
next edge of a prescribed color $\gamma$. If $s\geq r+2$, then,
conditional on $Y_r=\vec e$, we expose $Y_{s-1}=(x,y)$ and use the
same fact at the last step. Thus
\begin{equation}\label{eq:conditional-color-position}
 \Pp\bigl(c(Y_s)=\gamma\mid Y_r=\vec e\bigr)
 \leq \frac1{\alpha n-1}=O(n^{-1}).
\end{equation}
The left-hand side may be zero. Averaging over all possible values of
$Y_{s-1}$ causes no loss, since their conditional probabilities sum
to one. The reversed chain gives the same estimate when $s<r$.
The same last-step argument shows that, after fixing a vertex at one
vertex position, the conditional probability of a prescribed vertex at
any other vertex position is $O(n^{-1})$. The same holds after fixing a
directed edge, provided that the target position is not one of its two
endpoint positions.

Fix two distinct vertex positions. The first prescribed vertex occurs
with probability $d(u)/(2m)=O(n^{-1})$, and the preceding conditional
estimate gives another factor $O(n^{-1})$ for the second one. Summing
over the $O(\ell^2)$ pairs of positions proves
\eqref{eq:two-vertices}.

Write $k=|E_\gamma|$ and fix a vertex position and an edge position.
If the vertex is an endpoint of the edge, properness shows that the
joint probability is zero unless $v\in V(E_\gamma)$, and is then at
most $1/(2m)=O(n^{-2})$. Otherwise, condition further on the actual
directed edge of color $\gamma$. Its position has one of the
$2k$ possible values, each of probability $1/(2m)$, and the
conditional probability of reaching $v$ at the prescribed vertex
position is $O(n^{-1})$. Hence this case contributes
$O(k/(mn))=O(k/n^3)$. Summing over the possible positions proves
\eqref{eq:vertex-color}.

Finally, suppose that $|E_\gamma|\leq |E_{\gamma'}|$, and fix two
distinct edge positions. The probability of seeing $\gamma$ at its
prescribed position is $|E_\gamma|/m$. Conditional on its actual
directed edge, \eqref{eq:conditional-color-position}, applied
forward or backward as appropriate, bounds the probability of seeing
$\gamma'$ at the other position by $O(n^{-1})$. This argument also
covers nonadjacent positions: one averages over the state immediately
before the second position. Summing over the $O(\ell^2)$ ordered
pairs of positions gives
\eqref{eq:two-colors}.
\end{proof}

\subsection{Pseudorandom hypergraph matchings}\label{subsec:egj}

We use the following consequence of the weighted matching theorem of
Ehard, Glock, and Joos~\cite[Theorem~1.2]{EhardGlockJoos}.

\begin{theorem}[Ehard--Glock--Joos]\label{thm:egj}
Fix an integer $k\geq2$ and $\beta\in(0,1)$. There is $\zeta>0$ such
that, for every sufficiently large $D$, the following holds. Let
$\cH$ be a $k$-uniform hypergraph satisfying
\begin{equation}
 \Delta(\cH)\leq D,\qquad
 \Delta_2(\cH)\leq D^{1-\beta},\qquad
 |E(\cH)|\leq\exp(D^\zeta).
\end{equation}
Let $\mathcal W$ be a family of at most $\exp(D^\zeta)$ non-negative
weight functions on $E(\cH)$. If every $\omega\in\mathcal W$ satisfies
\begin{equation}\label{eq:egj-weight-condition}
 \omega(E(\cH))\geq
 D^{1+\beta}\max_{f\in E(\cH)}\omega(f),
\end{equation}
then $\cH$ contains a matching $\mathcal M$ such that, simultaneously
for every $\omega\in\mathcal W$,
\begin{equation}\label{eq:egj-conclusion}
 \omega(\mathcal M)
  =(1\pm D^{-\zeta})\frac{\omega(E(\cH))}{D}.
\end{equation}
\end{theorem}

\section{Graphs with linear minimum degree}
\label{sec:linear-minimum-degree}

The main result of this section has two parts. Part~(i) restates
Theorem~\ref{thm:dense-main}, while part~(ii) gives the sharper bound
for complete multipartite graphs under the same minimum-degree
hypothesis. Both proofs use the same stationary-path and
auxiliary-hypergraph mechanism, but they require different forms of
the packing lemma and different completion arguments.
We shall use the notation
\begin{equation}\label{eq:Q-definition}
 Q(G):=\max\{\nu(G),\Delta(G)/2\}.
\end{equation}

\begin{theorem}[Linear minimum degree]\label{thm:linear-minimum-degree}
For every $0<\alpha<1$ and every $\varepsilon>0$, there exists
$n_0=n_0(\alpha,\varepsilon)$ such that the following holds. Let $G$ be
a graph on $n\geq n_0$ vertices with $\delta(G)\geq\alpha n$, and let
$c$ be a proper edge-coloring of $G$.
\begin{enumerate}
\item[(i)] One has
\begin{equation}\label{eq:sharp-dense-upper}
 \rpc(G,c)\leq(1+\varepsilon)\frac n2.
\end{equation}
\item[(ii)] If $G$ is complete multipartite, then
\begin{equation}\label{eq:dense-multipartite-upper}
 \rpc(G,c)\leq(1+\varepsilon)Q(G).
\end{equation}
\end{enumerate}
\end{theorem}

Here is a guide to the proof. The next subsection develops the basic rainbow linear-forest packing in the form used for complete multipartite graphs.
The part structure then provides short direct connectors between the components of each forest. 
For part~(i), a general dense graph may contain large pieces joined by relatively few edges.
We first separate these pieces and organize the edges between them. We then adapt the same packing mechanism so that each forest avoids prescribed vertices and colors.
Finally, we join the resulting pieces and cover the few remaining edges.

\subsection{Rainbow linear-forest packing}\label{subsec:forest-packing}

We first develop the form of the packing used in the complete
multipartite case.

Fix $\alpha,\eta>0$. Let $(G,c)$ be a properly edge-colored graph on
$n$ vertices with $\delta(G)\geq\alpha n$. All constants in this
subsection may depend on $\alpha$ and $\eta$.

A \textbf{linear forest} is a graph whose connected components are
paths. It is \textbf{globally rainbow} if all its edges have pairwise
distinct colors, including edges in different components. Our aim is to
construct edge-disjoint globally rainbow linear forests
$F_1,\ldots,F_q$. Every component of each $F_i$ will be an
$\ell$-edge path. In the completion step, these components will be
joined using vertices and colors reserved for the index $i$.

The natural scale for $q$ is $Q(G)$. Since
$\delta(G)\geq\alpha n$, we have
\begin{equation}\label{eq:Q-linear}
\frac{\alpha n}{2}\leq\frac{\Delta(G)}2\leq Q(G)\leq\frac n2.
\end{equation}
We put
\begin{equation}\label{eq:q}
q:=\left\lceil\left(1+\frac{\eta}{3}\right)Q(G)\right\rceil.
\end{equation}
In particular, $q=\Theta(n)$.

We now choose the reserve densities. Take $\sigma,\tau>0$ sufficiently
small in terms of $\alpha$ and $\eta$, and then choose $\ell$
sufficiently large so that
\begin{equation}\label{eq:hierarchy}
0<\sigma,\tau\ll\eta,\alpha,
\qquad \frac1\ell\ll\sigma\tau^2\alpha\eta.
\end{equation}
For every $i\in[q]$, assign each vertex independently to exactly one
of $W_i^1$, $W_i^2$, and the unreserved set, with probabilities
$\sigma$, $\sigma$, and $1-2\sigma$, respectively. Independently,
assign each color to exactly one of $R_i^1$, $R_i^2$, $R_i^3$, and
the unreserved set, with probabilities $\tau$, $\tau$, $\tau$, and
$1-3\tau$, respectively. All assignments are independent for
different vertices, colors, and indices.

Put $\phi=1-2\sigma$ and $\theta=1-3\tau$, the probabilities that a
vertex and a color are unreserved for a fixed index. A pair $(i,P)$
is \textbf{admissible} if every vertex and color of $P$ is unreserved
for $i$.

The factors $1/\phi$ and $1/\theta$ will appear when we estimate the
degrees of the vertices $V_{i,v}$ and $C_{i,\gamma}$ in the auxiliary
hypergraph. By taking $\sigma$ and $\tau$ sufficiently small and then
$\ell$ sufficiently large, we may fix $\xi>0$ such that
\begin{equation}\label{eq:slack}
\frac{1+1/\ell}{\phi(1+\eta/3)}<1-\xi,
\qquad
\frac1{\theta(1+\eta/3)}<1-\xi.
\end{equation}
These inequalities ensure that these degrees are smaller, by a fixed
proportion, than the degrees of the vertices $E_e$. This is
needed when applying Theorem~\ref{thm:egj}.

\begin{lemma}[Regularized auxiliary hypergraph]\label{lem:regularisation}
With probability $1-o(1)$ over the choice of the reserves, there exists a
simple $(3\ell+1)$-uniform hypergraph $\cH$ whose vertex set consists of
the three disjoint families
\begin{equation}
 \{E_e:e\in E(G)\},\qquad
 \{V_{i,v}:i\in[q],v\in V(G)\},\qquad
 \{C_{i,\gamma}:i\in[q],\gamma\in\col(G)\}.
\end{equation}
Here $E_e$ records the use of the edge $e$, while $V_{i,v}$ and
$C_{i,\gamma}$ record the use of the vertex $v$ and the color
$\gamma$ by the forest with index $i$. Every hyperedge corresponds to
an admissible pair $(i,P)$ and is
\begin{equation}\label{eq:auxiliary-edge}
 \{E_e:e\in E(P)\}
 \cup\{V_{i,v}:v\in V(P)\}
 \cup\{C_{i,\gamma}:\gamma\in\col(P)\}.
\end{equation}
The edge-record degrees are asymptotically equal and have order
$n^\ell$; uniformly over $e,f\in E(G)$,
\begin{equation}\label{eq:degree-edge}
 d_{\cH}(E_e)=\Theta(n^\ell),
 \qquad d_{\cH}(E_e)=(1+o(1))d_{\cH}(E_f).
\end{equation}
Uniformly over $e\in E(G)$, $i\in[q]$, $v\in V(G)$, and
$\gamma\in\col(G)$, the following estimates hold:
\begin{align}
 d_{\cH}(V_{i,v})
  &=(1+o(1))d_{\cH}(E_e)\frac{\ell+1}{\ell}
      \frac{d(v)}{2q\phi},
  \label{eq:degree-vertex}\\
 d_{\cH}(C_{i,\gamma})
  &=(1+o(1))d_{\cH}(E_e)\frac{|E_\gamma|}{q\theta}.
  \label{eq:degree-color}
\end{align}
Estimate~\eqref{eq:degree-vertex} applies when $v$ is unreserved for $i$.
Estimate~\eqref{eq:degree-color} applies when $\gamma$ is unreserved for
$i$ and $|E_\gamma|\geq n^{1/2}$. If $v$ is reserved for $i$, then
$d_{\cH}(V_{i,v})=0$. If $\gamma$ is reserved for $i$, then
$d_{\cH}(C_{i,\gamma})=0$. If $|E_\gamma|<n^{1/2}$, then
$d_{\cH}(C_{i,\gamma})=o(n^\ell)$ uniformly over $i$. Moreover,
\begin{equation}\label{eq:auxiliary-codegree}
 \Delta_2(\cH)=O(n^{\ell-1}).
\end{equation}
\end{lemma}

\begin{proof}
We first choose the reserve sets and then, conditional on them, sample the
hyperedges. Let $C=C(\alpha,\ell)$ be the upper-bound constant in
\eqref{eq:path-weight-bound}, and fix a constant $0<a\leq C^{-1}$. For every
admissible pair $(i,P)$, include the hyperedge in
\eqref{eq:auxiliary-edge} independently with probability
$a n^{\ell+1}w(P)$. This is a valid probability because
$a n^{\ell+1}w(P)\leq aC\leq1$.

Put $\kappa=\phi^{\ell+1}\theta^\ell$. Every path in $\cP$ has
$\ell+1$ distinct vertices and $\ell$ distinct colors. Hence a fixed
pair $(i,P)$ is admissible with probability $\kappa$. For a fixed edge
$e$, taking expectation over both random choices gives
\begin{equation}\label{eq:unconditional-edge-degree}
 \begin{aligned}
 \E d_{\cH}(E_e)
  &=\sum_{i=1}^q\sum_{\substack{P\in\cP\\e\in E(P)}}
       \kappa a n^{\ell+1}w(P)\\
  &=(1+o(1))q\kappa a n^{\ell+1}\frac{\ell}{m}.
 \end{aligned}
\end{equation}
The second equality is precisely the edge marginal
\eqref{eq:edge-marginal} from Lemma~\ref{lem:stationary}. We therefore
define
\begin{equation}\label{eq:D0-definition}
 D_0:=q\kappa a n^{\ell+1}\frac{\ell}{m}.
\end{equation}
By \eqref{eq:Q-linear}, $q=\Theta(n)$. Also $m=\Theta(n^2)$ because
$\delta(G)\geq\alpha n$. Thus $D_0=\Theta(n^\ell)$.

The resulting hypergraph is simple. Indeed, if $P\neq P'$, then
$E(P)\neq E(P')$, since the edge set of a simple path determines it
up to reversal. If $P=P'$ but $i\neq i'$, then the two hyperedges
contain different vertices of the form $V_{i,v}$.

We next show that the reserve sets can be fixed so that all conditional
expected degrees have the required values. For $i\in[q]$ and
$v\in V(G)$, let
$Z_{i,v}$ be the sum of $w(P)$ over all paths $P$ containing $v$ for
which every vertex in $V(P)\setminus\{v\}$ and every color in
$\col(P)$ are unreserved for $i$. The reserve status of $v$ itself is
not tested in this definition. Similarly, let $Z_{i,\gamma}$ be the
sum of $w(P)$ over all paths $P$ using $\gamma$ for which every vertex
in $V(P)$ and every color in $\col(P)\setminus\{\gamma\}$ are
unreserved for $i$. Put $k_\gamma=|E_\gamma|$. The vertex and color
marginals \eqref{eq:vertex-marginal} and
\eqref{eq:color-marginal} give
\begin{equation}\label{eq:indexed-means}
 \begin{aligned}
 \E Z_{i,v}
   &=(1+o(1))\frac{\kappa}{\phi}
      \frac{(\ell+1)d(v)}{2m}=\Theta(n^{-1}),\\
 \E Z_{i,\gamma}
   &=(1+o(1))\frac{\kappa}{\theta}
      \frac{\ell k_\gamma}{m}=\Theta(k_\gamma n^{-2}).
 \end{aligned}
\end{equation}

We apply the bounded-differences inequality
\eqref{eq:mcdiarmid-standard} to these sums. Changing the reserve
status of a vertex $u\neq v$ changes $Z_{i,v}$ by at most the total
weight of paths containing both $u$ and $v$. This is $O(n^{-2})$ by
\eqref{eq:two-vertices}. Changing the reserve status of a color
$\gamma$ changes $Z_{i,v}$ by at most
$O(\mathbf1_{\{v\in V(E_\gamma)\}}n^{-2}+k_\gamma n^{-3})$ by
\eqref{eq:vertex-color}. Properness gives
$\sum_\gamma\mathbf1_{\{v\in V(E_\gamma)\}}=d(v)$ and
$\max_\gamma k_\gamma\leq n/2$, while
$\sum_\gamma k_\gamma=m=O(n^2)$. Hence
\begin{equation}\label{eq:vertex-influence-squares}
 \begin{aligned}
 &\sum_{u\neq v}O(n^{-2})^2+
 \sum_\gamma
 O\left(\frac{\mathbf1_{\{v\in V(E_\gamma)\}}}{n^2}
       +\frac{k_\gamma}{n^3}\right)^2\\
 &\qquad=O\left(\frac{n}{n^4}+\frac{d(v)}{n^4}
       +\frac{\max_\gamma k_\gamma\sum_\gamma k_\gamma}{n^6}\right)
 =O(n^{-3}).
 \end{aligned}
\end{equation}

Now fix a color $\gamma$ and put $k=k_\gamma$. By
\eqref{eq:vertex-color}, changing the reserve status of a vertex $u$
changes $Z_{i,\gamma}$ by at most
$O(\mathbf1_{\{u\in V(E_\gamma)\}}n^{-2}+kn^{-3})$. Since
$|V(E_\gamma)|=2k$ and $k\leq n/2$, the squared vertex influences
sum to
\begin{equation}\label{eq:color-vertex-influence-squares}
 O\left(\frac{|V(E_\gamma)|}{n^4}+\frac{k^2}{n^5}\right)
 =O(kn^{-4}).
\end{equation}
For $\gamma'\neq\gamma$, let $b_{\gamma'}$ be the total weight of
paths using both $\gamma$ and $\gamma'$. Equation
\eqref{eq:two-colors} gives
$\max_{\gamma'\neq\gamma}b_{\gamma'}=O(kn^{-3})$. Since every path
using $\gamma$ uses at most $\ell-1$ other colors,
$\sum_{\gamma'\neq\gamma}b_{\gamma'}=O(kn^{-2})$ by
\eqref{eq:color-marginal}. Therefore
\begin{equation}\label{eq:color-influence-squares}
 \sum_{\gamma'\neq\gamma}b_{\gamma'}^2
 \leq \left(\max_{\gamma'\neq\gamma}b_{\gamma'}\right)
       \sum_{\gamma'\neq\gamma}b_{\gamma'}
 =O(k^2n^{-5})=O(kn^{-4}).
\end{equation}
Together with \eqref{eq:color-vertex-influence-squares}, this gives
total squared influence $O(kn^{-4})$ for $Z_{i,\gamma}$.

Set $\rho=n^{-1/10}$. The bounded-differences inequality
\eqref{eq:mcdiarmid-standard} and
\eqref{eq:indexed-means} give
\begin{equation}
 \begin{aligned}
 \Pp\bigl(|Z_{i,v}-\E Z_{i,v}|>\rho\E Z_{i,v}\bigr)
   &\leq\exp(-\Omega(n^{4/5})),\\
 \Pp\bigl(|Z_{i,\gamma}-\E Z_{i,\gamma}|>
                  \rho\E Z_{i,\gamma}\bigr)
   &\leq\exp(-\Omega(k n^{-1/5})).
 \end{aligned}
\end{equation}
For $k\geq n^{1/2}$, the second bound is at most
$\exp(-\Omega(n^{3/10}))$. There are at most $m=O(n^2)$ colors,
since every color is used by an edge. A union bound gives these estimates
simultaneously for every $(i,v)$ and every $(i,\gamma)$ with
$k_\gamma\geq n^{1/2}$.

For $e\in E(G)$, let $Y_{i,e}$ be the sum of $w(P)$ over all paths
$P$ containing $e$ that are admissible for $i$. The variables
$Y_{i,e}$ are independent for different $i$. By
\eqref{eq:edge-marginal}, each is $O(n^{-2})$, and
\begin{equation}\label{eq:edge-reserve-mean}
 \E\sum_{i=1}^qY_{i,e}
  =q\kappa\sum_{\substack{P\in\cP\\e\in E(P)}}w(P)
  =(1+o(1))q\kappa\frac{\ell}{m}.
\end{equation}
Since the last quantity is $\Theta(n^{-1})$,
\eqref{eq:hoeffding-standard}, with relative error $\rho$, gives failure
probability
$\exp(-\Omega(n\rho^2))=\exp(-\Omega(n^{4/5}))$. Another union bound
gives, uniformly over $e\in E(G)$,
\begin{equation}\label{eq:edge-reserve-concentration}
 \sum_{i=1}^qY_{i,e}
  =(1\pm\rho)(1+o(1))q\kappa\frac{\ell}{m}.
\end{equation}
Fix reserve sets for which this estimate and the preceding indexed
estimates hold.

After the reserve sets are fixed, the hyperedges are sampled
independently. Their conditional expected degrees are
\begin{equation}\label{eq:conditional-degree-means}
 \begin{aligned}
 \E d_{\cH}(E_e)
   &=a n^{\ell+1}\sum_{i=1}^qY_{i,e}
     =(1+o(1))D_0,\\
 \E d_{\cH}(V_{i,v})
   &=a n^{\ell+1}Z_{i,v}
     =(1+o(1))D_0\frac{\ell+1}{\ell}
          \frac{d(v)}{2q\phi},\\
 \E d_{\cH}(C_{i,\gamma})
   &=a n^{\ell+1}Z_{i,\gamma}
     =(1+o(1))D_0\frac{k_\gamma}{q\theta}
     \quad\text{if }k_\gamma\geq n^{1/2}.
 \end{aligned}
\end{equation}
The second line applies when $v$ is unreserved for $i$, and the third
when $\gamma$ is unreserved for $i$. Equation
\eqref{eq:chernoff-standard}, followed by a union bound, now gives
\eqref{eq:degree-edge}--\eqref{eq:degree-color}. If
$v$ or $\gamma$ is reserved for $i$, the corresponding degree is zero. If
$k_\gamma<n^{1/2}$, then \eqref{eq:color-marginal} gives
\begin{equation}
 \E d_{\cH}(C_{i,\gamma})
 \leq a n^{\ell+1}
       \sum_{\substack{P\in\cP\\\gamma\in\col(P)}}w(P)
 =O(k_\gamma n^{\ell-1})=o(D_0).
\end{equation}
Since $k_\gamma\geq1$ and $\ell\geq3$, we have
$k_\gamma n^{\ell-1}\gg\log n$. The upper-tail bound
\eqref{eq:chernoff-upper}, with $s$ a sufficiently large multiple of
$k_\gamma n^{\ell-1}$, followed by a union bound, therefore gives
$d_{\cH}(C_{i,\gamma})=O(k_\gamma n^{\ell-1})=o(D_0)$
uniformly over all such $i$ and $\gamma$.

It remains to bound the codegrees. Fix distinct edges $e,f$ and two
distinct edge positions $r,s$. Conditional on an orientation of $e$
being traversed at position $r$, expose the walk up to position $s-1$,
using the reversed chain if $s<r$. Given the directed edge at position
$s-1$, the probability that the next edge is $f$ is at most
$2/(\alpha n-1)=O(n^{-1})$. Summing over the $O(\ell^2)$ choices of
$r,s$ shows that the total weight of paths containing both $e$ and $f$
is $O(n^{-3})$. The total weight of paths containing a fixed edge is
$O(n^{-2})$ by \eqref{eq:edge-marginal}. Moreover,
\eqref{eq:two-vertices}--\eqref{eq:two-colors} show that the total
weight of paths containing any two prescribed vertices or colors is
$O(n^{-2})$; the same is true for one prescribed vertex and one
prescribed color.

Let $A$ and $B$ be two distinct vertices of $\cH$ of the form
$V_{i,v}$ or $C_{i,\gamma}$, with the same index $i$. The preceding
bounds and the sampling probability $a n^{\ell+1}w(P)$ give
\begin{equation}\label{eq:codegree-means}
 \begin{aligned}
 \E d_{\cH}(E_e,E_f)&=O(q a n^{\ell+1}n^{-3}),\\
 \E d_{\cH}(E_e,V_{i,v})&=O(a n^{\ell+1}n^{-2}),\\
 \E d_{\cH}(E_e,C_{i,\gamma})&=O(a n^{\ell+1}n^{-2}),\\
 \E d_{\cH}(A,B)&=O(a n^{\ell+1}n^{-2}).
 \end{aligned}
\end{equation}
Two indexed vertices with different indices have codegree zero. Since
$q=\Theta(n)$ and $D_0=\Theta(n^\ell)$, every quantity in
\eqref{eq:codegree-means} is $O(D_0/n)$. The auxiliary vertex set has
polynomial size, while $D_0/n=\Theta(n^{\ell-1})\gg\log n$. The
upper-tail bound \eqref{eq:chernoff-upper}, with $s$ a sufficiently
large multiple of $D_0/n$, followed by a union bound over all pairs,
therefore gives
$\Delta_2(\cH)=O(D_0/n)=O(n^{\ell-1})$. This proves
\eqref{eq:auxiliary-codegree}.
\end{proof}

\begin{lemma}[Dense rainbow packing]\label{lem:forest-packing}
Fix $\alpha,\eta>0$, choose $\sigma,\tau,\ell$ as in
\eqref{eq:hierarchy}--\eqref{eq:slack}, and let $q$ be given by
\eqref{eq:q}. If $(G,c)$ is a properly edge-colored $n$-vertex graph with
$\delta(G)\geq\alpha n$, then, with probability $1-o(1)$ over the choice
of the reserves, there are edge-disjoint globally rainbow linear forests
$F_1,\ldots,F_q$, each of whose components is an $\ell$-edge path, such
that the residual graph
\begin{equation}
 R:=G-\bigcup_{i=1}^qE(F_i)
\end{equation}
satisfies
\begin{equation}\label{eq:packing-residual}
 \Delta(R)=o(n),\qquad \mu_c(R)=o(n).
\end{equation}
Each $F_i$ avoids the vertices in $W_i^1\cup W_i^2$ and the colors
in $R_i^1\cup R_i^2\cup R_i^3$.
\end{lemma}

\begin{proof}
On the event supplied by Lemma~\ref{lem:regularisation}, fix an auxiliary
hypergraph $\cH$ with the stated properties.
By \eqref{eq:slack}, \eqref{eq:degree-vertex}, and
\eqref{eq:degree-color}, every vertex of $\cH$ of the form
$V_{i,v}$ or $C_{i,\gamma}$ has degree at most
$(1-\xi/2)d_{\cH}(E_e)$ for every $e\in E(G)$, because
\begin{equation}
 \frac{d(v)}2\leq Q(G),
 \qquad |E_\gamma|\leq\nu(G)\leq Q(G).
\end{equation}
The same conclusion is immediate for reserved vertices and for colors
with fewer than $n^{1/2}$ edges. Put $\Delta:=\Delta(\cH)$. It follows
from \eqref{eq:degree-edge} that
\begin{equation}
 \Delta=\max_{e\in E(G)}d_{\cH}(E_e)=\Theta(n^\ell),
 \qquad d_{\cH}(E_e)=(1+o(1))\Delta
\end{equation}
uniformly over $e$.
We apply Theorem~\ref{thm:egj} with $k=3\ell+1$ and
$\beta=1/(4\ell)$. Then
\begin{equation}
 \Delta_2(\cH)=O(n^{\ell-1})
  \leq \Delta^{1-\beta}
\end{equation}
for large $n$.

For every host vertex $v$, define a weight on auxiliary edges by
$\omega_v(i,P)=d_P(v)\in\{0,1,2\}$. For every color $\gamma$ with
$|E_\gamma|\geq n^{1/2}$, define
$\omega_\gamma(i,P)=\mathbf1_{\{\gamma\in\col(P)\}}$. Since the
vertices $E_e$ have degree $(1+o(1))\Delta$,
\begin{equation}\label{eq:weight-totals}
 \omega_v(E(\cH))=(1+o(1))d(v)\Delta,
 \qquad
 \omega_\gamma(E(\cH))=(1+o(1))|E_\gamma|\Delta.
\end{equation}
Here the first identity counts each auxiliary path according to its
number of incidences with $v$, and the second uses that every path in
$\cP$ is rainbow. Since $\beta=1/(4\ell)$ and
$\Delta=\Theta(n^\ell)$, we have
$\Delta^{1+\beta}=\Theta(n^{\ell+1/4})$, whereas the vertex-weight
totals have order $n^{\ell+1}$ and the relevant color-weight totals
have order at least $n^{\ell+1/2}$. Thus every weight satisfies
\eqref{eq:egj-weight-condition}. The auxiliary edge set and the weight
family have polynomial size in $n$. Theorem~\ref{thm:egj} therefore
gives a matching $\mathcal M$ such that, simultaneously for all the
weights $\omega_v$ and $\omega_\gamma$ defined above,
\begin{equation}\label{eq:forest-packing-egj}
 \omega(\mathcal M)
  =(1\pm\Delta^{-\zeta})\frac{\omega(E(\cH))}{\Delta}.
\end{equation}

For $i\in[q]$, let $F_i$ be the union of the paths $P$ for which
$(i,P)\in\mathcal M$. Since $\mathcal M$ is a matching, no two selected
hyperedges contain the same $E_e$, so all selected path edges are
distinct. For each fixed $i$, no two selected hyperedges contain the
same $V_{i,v}$ or $C_{i,\gamma}$, so the paths in $F_i$ are
vertex-disjoint and their union is globally rainbow. Moreover, $F_i$
uses no vertex of $W_i^1\cup W_i^2$
and no color from $R_i^1\cup R_i^2\cup R_i^3$. Let $R$ be the graph
formed by the host edges not selected into any $F_i$. From
\eqref{eq:forest-packing-egj} and \eqref{eq:weight-totals},
\begin{equation}
 \Delta(R)=o(n),
 \qquad
 \mu_c(R)=o(n).
\end{equation}
Indeed, the vertex weights leave $o(n)$ incident edges at every vertex.
For each color $\gamma$ with $|E_\gamma|\geq n^{1/2}$, the color
weight leaves $o(n)$ of its edges in $R$. Every other color has at
most $n^{1/2}$ edges from the outset.
This proves the lemma.
\end{proof}

\subsection{Complete multipartite graphs}\label{subsec:dense-multipartite}

We first prove part~(ii). The next three lemmas decompose the residual
graph, provide connectors for the packed forests, and extend small rainbow
matchings in complete multipartite graphs.

\begin{lemma}[Residual matching decomposition]\label{lem:residual-matchings}
Let $(G,c)$ be properly edge-colored, let $J\subseteq G$, and put
\begin{equation}
 D:=\max\{\Delta(J),\mu_c(J)\}.
\end{equation}
Then $E(J)$ can be partitioned into at most $3D$ globally rainbow
matchings.
\end{lemma}

\begin{proof}
The assertion is immediate when $D=0$, so assume $D\geq1$.
Form a conflict graph on $E(J)$ by joining two edges if they share an
endpoint or a color. Its maximum degree is at most $3D-3$. A greedy
coloring uses at most $3D-2$ colors, so its color classes form at most
$3D$ independent sets. Each independent set is a matching whose edges
have pairwise distinct colors.
\end{proof}

We next record the reserve property used to join the components of the
packed forests.
Recall that, for each $i\in[q]$, $W_i^1,W_i^2$ are the reserved vertex
sets and $R_i^1,R_i^2,R_i^3$ are the reserved color sets chosen in
Subsection~\ref{subsec:forest-packing}.

\begin{lemma}[Connector reserves]\label{lem:connector-reserves}
Suppose that $G$ is complete multipartite. With probability $1-o(1)$,
the following assertions hold simultaneously for every $i\in[q]$.
\begin{enumerate}
\item If distinct $x,y$ lie in the same part $A$, then at least
$\sigma\tau^2(n-|A|)/2$ vertices $z\notin A$ satisfy
\begin{equation}\label{eq:two-edge-reserve}
 z\in W_i^1,\qquad c(xz)\in R_i^1,\qquad c(zy)\in R_i^2.
\end{equation}
\item If $x,y$ lie in different parts $A,B$, then at least
$\sigma\tau^2(n-|A|-|B|)/2$ vertices $z\notin A\cup B$ satisfy
\eqref{eq:two-edge-reserve}.
\item If $A,B$ are distinct parts and $x\in A$, then at least
$\sigma\tau|B|/2$ vertices $u\in B$ satisfy
\begin{equation}\label{eq:first-three-edge-reserve}
 u\in W_i^1,\qquad c(xu)\in R_i^1.
\end{equation}
For every such $u$ and every $y\in B\setminus\{u\}$, at least
$\sigma\tau^2|A|/2$ vertices $v\in A$ satisfy
\begin{equation}\label{eq:second-three-edge-reserve}
 v\in W_i^2,\qquad c(uv)\in R_i^2,
 \qquad c(vy)\in R_i^3.
\end{equation}
\end{enumerate}
Each assertion is required only when its corresponding candidate set,
namely $V(G)\setminus A$, $V(G)\setminus(A\cup B)$, $B$, or $A$, has
size at least $\alpha n/2$.
\end{lemma}

\begin{proof}
Fix $i,x,y$ and a candidate set $Z$, and let $I_z$ be the indicator of
\eqref{eq:two-edge-reserve} for $z\in Z$. Then
$\E I_z=\sigma\tau^2$. Along each of the stars centered at $x$ and $y$,
the colors are injective. Hence every $I_z$ depends on at most two other
indicators: dependence can arise only from a cross-equality between the
two edge colors. The resulting dependency graph has maximum degree at
most two. Its path and cycle components admit a proper three-coloring
whose color classes differ in size by at most one. Within each class the
indicators are independent, so Chernoff's inequality shows that, except
with probability $\exp(-\Omega(n))$, their total is at least
$\sigma\tau^2|Z|/2$ whenever $|Z|\geq\alpha n/2$.

The indicators in \eqref{eq:first-three-edge-reserve} are independent.
For the second estimate in item~3, first fix $i,A,B,u,y$. The same
dependency-graph argument over $v\in A$ applies, whether or not $u$
satisfies \eqref{eq:first-three-edge-reserve}. Intersecting the two
estimates gives item~3. There are $O(n^4)$ choices of the index, parts,
and prescribed vertices. A union bound proves all three assertions
simultaneously.
\end{proof}

Finally, we show that small rainbow matchings admit a stronger completion.

\begin{lemma}[Matching extension]\label{lem:matching-extension}
For every fixed $\alpha>0$ and all sufficiently large $n$, the
following holds. If $G$ is a complete multipartite graph on $n$
vertices with $\delta(G)\geq\alpha n$, then every rainbow matching of
size at most
\begin{equation}
 t_0:=\left\lfloor\frac{\alpha n}{100}\right\rfloor
\end{equation}
is contained in a rainbow simple path of $G$.
\end{lemma}

\begin{proof}
Put $\delta:=\delta(G)$, and let $A_0$ be a largest part, so
$\delta=n-|A_0|$. We first record the potential connectors between two
distinct vertices $x,y$. If they lie in the same part $A$, then
every $z\in V(G)\setminus A$ gives a potential connector $xzy$, and
there are $n-|A|\geq\delta$ such potential connectors.

Suppose instead that $x,y$ lie in distinct parts $A,B$, respectively,
and put $r=n-|A|-|B|$. Then
\begin{equation}\label{eq:connector-dichotomy}
 r+|B|=n-|A|\geq\delta,
 \qquad
 r+|A|=n-|B|\geq\delta.
\end{equation}
If $r\geq\delta/2$, every vertex outside $A\cup B$ gives a potential
connector $xzy$, so there are at least $\delta/2$ such potential
connectors. Otherwise, \eqref{eq:connector-dichotomy} gives
$|A|,|B|>\delta/2$, and we take as potential connectors the paths
$xuvy$ with $u\in B$ and $v\in A$.

Let the matching have $t\leq t_0$ edges. The assertion is immediate
when $t\leq1$, so suppose that $t\geq2$. Order and orient its edges,
and join consecutive edges greedily, each time choosing one of the
potential connectors described above. Declare all matching vertices
and matching colors used from the outset. At any stage, let $U$
consist of these vertices together with the internal vertices of the
connectors already chosen, and let $\mathcal C$ consist of the
matching colors together with the colors of the connectors already
chosen. Since fewer than $t$ connectors have been chosen,
\begin{equation}\label{eq:used-connector-resources}
 |U|<4t,
 \qquad
 |\mathcal C|<4t.
\end{equation}

Suppose first that the current endpoints admit potential two-edge
connectors $xzy$. We must exclude those for which $z\in U$,
$c(xz)\in\mathcal C$, or $c(zy)\in\mathcal C$. Properness implies
that each color in $\mathcal C$ excludes at most one choice for $z$
at each endpoint. Thus fewer than
$|U|+2|\mathcal C|<12t$ potential connectors are excluded. There are
at least $\delta/2$ potential connectors, while
\begin{equation}
 12t\leq12t_0<\frac{\alpha n}{2}\leq\frac{\delta}{2}.
\end{equation}
Hence an available potential connector remains.

Suppose next that the current endpoints admit potential three-edge
connectors $xuvy$, where $u\in B$ and $v\in A$. First choose
$u\in B\setminus U$ with $c(xu)\notin\mathcal C$. Fewer than
$|U|+|\mathcal C|<8t$ choices for $u$ are excluded. Moreover,
$8t\leq8t_0<\alpha n/2<|B|$, so such a vertex $u$ exists. Having fixed
$u$, choose $v\in A\setminus U$ such that
\begin{equation}
 c(uv),c(vy)\notin\mathcal C,
 \qquad
 c(vy)\neq c(xu).
\end{equation}
Properness implies that each color in $\mathcal C$ excludes at most
one choice for $v$ through the edge $uv$ and at most one choice
through the edge $vy$. The last condition excludes at most one
additional choice. Hence fewer than
$|U|+2|\mathcal C|+1<12t+1$ choices for $v$ are excluded. Since
$|A|>\delta/2\geq\alpha n/2$, whereas
$12t_0+1<\alpha n/2$ for all sufficiently large $n$, such a vertex
$v$ exists. Thus an available potential connector again remains.

Choose an available potential connector and add it to the path.
Properness makes adjacent connector edges differently colored. For
a three-edge connector, the condition $c(vy)\neq c(xu)$ ensures that
its two non-adjacent outer edges also have different colors.
Therefore each step preserves simplicity and the rainbow property.
Repeating the argument joins all matching edges into a rainbow simple
path.
\end{proof}

\begin{proof}[Proof of Theorem~\ref{thm:linear-minimum-degree}(ii)]
It is enough to consider $0<\varepsilon\leq1$. Set $\eta=\varepsilon$
and choose $\sigma,\tau,\ell$ as in
\eqref{eq:hierarchy}--\eqref{eq:slack}. The events in
Lemmas~\ref{lem:forest-packing} and~\ref{lem:connector-reserves} each
have probability $1-o(1)$. Fix a reserve realization for which both
events occur, and choose the resulting globally rainbow linear forests
$F_1,\ldots,F_q$ and residual graph $R$ satisfying
\eqref{eq:packing-residual}.

Put $\delta:=\delta(G)$. Each $F_i$ has at most $n/(\ell+1)$
components. Order and orient its non-empty components. Suppose that
the current endpoint $x$ and the initial vertex $y$ of the next
component lie in parts $A$ and $B$. If $A=B$, then
$n-|A|\geq\delta$, so we use a connector $xzy$ supplied by
Lemma~\ref{lem:connector-reserves}. Suppose that $A\neq B$ and put
$r=n-|A|-|B|$. If $r\geq\delta/2$, we again use $xzy$. If
$r<\delta/2$, then \eqref{eq:connector-dichotomy} gives
$|A|,|B|>\delta/2$, and we use $xuvy$, where $u\in B$ and $v\in A$,
through \eqref{eq:first-three-edge-reserve} and
\eqref{eq:second-three-edge-reserve}. Thus every set from which an
internal vertex is chosen has size at least
$\delta/2\geq\alpha n/2$.

The internal connector vertices lie in $W_i^1\cup W_i^2$, which is
disjoint from $V(F_i)$, while their colors lie in the pairwise disjoint
sets $R_i^1,R_i^2,R_i^3$, which avoid the colors of $F_i$. At a given
stage, let $U$ be the set of internal connector vertices already used and
let $\mathcal C$ be the set of connector colors already used. Since
fewer than $n/(\ell+1)$ connectors are needed,
\begin{equation}
 |U|<\frac{2n}{\ell+1},
 \qquad
 |\mathcal C|<\frac{3n}{\ell+1}.
\end{equation}
For a connector $xzy$, the loss is at most
$|U|+2|\mathcal C|$. For a connector $xuvy$, the choice of $u$ loses at
most $|U|+|\mathcal C|$ candidates, and, after $u$ is fixed, the choice
of $v$ loses at most $|U|+2|\mathcal C|+1$ candidates. Here properness
ensures that a forbidden color excludes at most one candidate at each
fixed endpoint, and the final term excludes equality of the two outer
colors. Thus every greedy choice excludes fewer than
\begin{equation}\label{eq:multipartite-connector-loss}
 \frac{8n}{\ell+1}+1
\end{equation}
candidates. Each available reserve pool has size at least
$\sigma\tau^2\alpha n/4$. By
\eqref{eq:hierarchy}, this is larger than
\eqref{eq:multipartite-connector-loss}. We may therefore choose all
connectors so that their internal vertices and colors are new. Hence
every non-empty $F_i$ is contained in one rainbow simple path, using
at most $q$ paths in total.

It remains to cover $R$. Put
$D_R:=\max\{\Delta(R),\mu_c(R)\}=o(n)$. By
Lemma~\ref{lem:residual-matchings}, $E(R)$ can be partitioned into at most
$3D_R$ globally rainbow matchings. Split each of them into blocks of
size at most $t_0$. Since $t_0\geq\alpha n/200$ for sufficiently
large $n$ and $e(R)\leq nD_R/2$, the total number of blocks is at most
\begin{equation}\label{eq:multipartite-residual-blocks}
 3D_R+\frac{e(R)}{t_0}
 \leq\left(3+\frac{100}{\alpha}\right)D_R=o(n).
\end{equation}
Lemma~\ref{lem:matching-extension} extends every block to a rainbow simple
path. The extra edges used by different extensions may overlap each other
or the paths already constructed; this is permitted in an edge cover.
Finally, $Q(G)\geq\alpha n/2$, so $o(n)=o(Q(G))$. Equations
\eqref{eq:q} and \eqref{eq:Q-linear} give
\begin{equation}
 q+o(n)
 \leq\left(1+\frac\eta3\right)Q(G)+1+o(n)
 \leq(1+\eta)Q(G).
\end{equation}
This proves part~(ii).
\end{proof}

\subsection{Packing with forbidden sets}
\label{subsec:robust-tools}

We now prepare the packing used for arbitrary dense graphs. We first prove
that the required short rainbow paths remain available after prescribed
vertices and colors are excluded. We then obtain the main packing result.

\begin{lemma}[Many short rainbow paths]\label{lem:robust-connections}
Fix constants $d>a>0$ and $b>0$. There are constants
$L=L(d,a,b)$ and $\zeta=\zeta(d,a,b)>0$, and an integer
$n_1=n_1(d,a,b)$, with the following property. Suppose that
$n\ge n_1$ and that $H$ is a properly edge-colored graph with
\begin{equation}
|V(H)|\le n,\qquad \delta(H)\ge dn,
\end{equation}
and suppose that every set $S\subseteq V(H)$ with
$|S|,|V(H)\setminus S|\geq an$ satisfies
\begin{equation}
e_H(S,V(H)\setminus S)\ge bn^2.
\label{eq:robust-cut}
\end{equation}
Then, for every pair of distinct vertices $x,y\in V(H)$, there is an integer
$h=h(x,y)$ with $2\le h\le L$ for which $H$ contains at least
\begin{equation}
\zeta n^{h-1}
\label{eq:many-connections}
\end{equation}
rainbow simple $x$--$y$ paths of length $h$.
\end{lemma}

\begin{proof}
Set $T:=\lceil2/b\rceil$ and $L:=T+2$. Define constants
$\kappa_0,\ldots,\kappa_T$ in advance by $\kappa_0=1$ and
\begin{equation}
 \kappa_{t+1}:=\min\left\{\kappa_t,
              \frac{b\kappa_t}{2(t+1)}\right\}
 \qquad(0\leq t<T).
\end{equation}
We construct increasing sets $U_0,\ldots,U_T$. At stage $t$, assign to
every $u\in U_t$ a length
$\ell(u)\in\{1,\ldots,t+1\}$ such that $u$ is the end of at least
\begin{equation}
 \kappa_t n^{\ell(u)-1}
\end{equation}
$x$--$u$ walks of length $\ell(u)$. For $t=0$, put $U_0=N_H(x)$ and
set $\ell(u)=1$ for every $u\in U_0$. Since
$|U_0|\geq dn>an$, the required property holds.

Suppose that $U_t$ has been constructed and
$|V(H)\setminus U_t|\geq an$. Since $U_t\supseteq U_0$, we also have
$|U_t|>an$, so \eqref{eq:robust-cut} applies. Define
\begin{equation}
 Z_t:=\{z\in V(H)\setminus U_t:
              d_H(z,U_t)\geq bn/2\}.
\end{equation}
Then
\begin{equation}
 bn^2\leq e_H(U_t,V(H)\setminus U_t)
        \leq |Z_t|n+bn^2/2,
\end{equation}
and hence $|Z_t|\geq bn/2$. For each $z\in Z_t$, pigeonhole among the at
most $t+1$ values of $\ell(u)$ on its neighbors in $U_t$. Choose a
value $q(z)\leq t+1$ for which this gives at least
\begin{equation}
 \frac{b\kappa_t}{2(t+1)}n^{q(z)}
\end{equation}
$x$--$z$ walks of length $q(z)+1$. Set
$U_{t+1}=U_t\cup Z_t$, retain the old assignments on $U_t$, and put
$\ell(z)=q(z)+1$ for $z\in Z_t$. The definition of $\kappa_{t+1}$
preserves the required lower bound for every vertex in $U_{t+1}$.

Every non-terminal step adds at least $bn/2$ vertices. If the first $T$
steps were all non-terminal, then
$|U_T|\geq|U_0|+Tbn/2>n$, a contradiction. Hence, for some $t\leq T$,
we obtain $|V(H)\setminus U_t|<an$. Therefore
\begin{equation}
 |N_H(y)\cap U_t|
  \geq d_H(y)-|V(H)\setminus U_t|>(d-a)n.
\end{equation}
Pigeonholing the values of $\ell(u)$ in this intersection and appending
the edge $uy$ gives, for some $2\leq h\leq t+2\leq L$, at least
$c n^{h-1}$ $x$--$y$ walks of length $h$, where
$c:=\min_{0\leq s\leq T}(d-a)\kappa_s/(s+1)>0$.

For fixed $h$, the walks with a repeated vertex number
$O_h(n^{h-2})$. After deleting them, any two edges of the same color
occupy non-consecutive positions. Fix these positions and choose all
internal vertices except one endpoint of the latter edge. The first
edge determines the color, and properness gives at most one choice for
the omitted vertex from the other endpoint of the latter edge. Thus the
walks with a repeated color also number $O_h(n^{h-2})$. Since
$h\leq L$, choosing $\zeta<c/2$ and then $n_1$ sufficiently large
proves \eqref{eq:many-connections}.
\end{proof}

\begin{lemma}[Short rainbow paths in reserves]
\label{lem:reserved-robust-connections}
Fix constants $d>a>0$, $b>0$, and $0<\sigma,\tau<1$, and let
$L=L(d,a,b)$ be given by Lemma~\ref{lem:robust-connections}. There are
constants $r_1=r_1(d,a,b,\sigma,\tau)>0$ and
$\zeta_*=\zeta_*(d,a,b,\sigma,\tau)>0$ such that the
following holds for every $0\le r\le r_1$ and all sufficiently large $n$.
Let $H$ be a properly edge-colored graph satisfying
\begin{equation}
 |V(H)|\le n,\qquad \delta(H)\ge dn,
\end{equation}
and suppose that every set $S\subseteq V(H)$ with
$|S|,|V(H)\setminus S|\ge an$ satisfies
\begin{equation}
 e_H(S,V(H)\setminus S)\ge bn^2.
\end{equation}

Let $A\subseteq V(H)$ and $B\subseteq\col(H)$ satisfy
\begin{equation}
 |A|+|B|\le rn.
 \label{eq:forbidden-load}
\end{equation}
Choose $W\subseteq V(H)\setminus A$ by including each vertex of
$V(H)\setminus A$ independently with probability $\sigma$. Independently,
choose $R\subseteq\col(H)\setminus B$ by including each color of
$\col(H)\setminus B$ with probability $\tau$.

Then, with probability $1-\exp(-\Omega(n))$, for every pair of distinct
vertices $x,y\in V(H)$ there is an integer $h$ with $2\le h\le L$ such
that $H$ contains at least
\begin{equation}
 \zeta_*n^{h-1}
 \label{eq:reserved-connections}
\end{equation}
rainbow $x$--$y$ paths of length $h$ whose internal vertices lie in
$W$ and whose colors lie in $R$.

We shall also use the following simultaneous version. Let
$H_1,\ldots,H_k$ be a bounded collection of pairwise vertex-disjoint
graphs, each satisfying the hypotheses above, and let $(A_t,B_t)$,
$t\in\mathcal T$, be $O(n)$ prescribed pairs satisfying
\eqref{eq:forbidden-load}. For each $t$, choose global reserves
$W_t\subseteq(\bigcup_{j=1}^kV(H_j))\setminus A_t$ and
$R_t\subseteq(\bigcup_{j=1}^k\col(H_j))\setminus B_t$ by the same
independent sampling rule. Then, with probability
$1-\exp(-\Omega(n))$, the conclusion holds for every $t\in\mathcal T$,
every $j\in[k]$, and every pair of distinct vertices $x,y\in V(H_j)$,
using the restrictions of $W_t$ and $R_t$ to $H_j$.
\end{lemma}

\begin{proof}
Fix two distinct vertices $x,y\in V(H)$. By
Lemma~\ref{lem:robust-connections}, there exist an integer $h$ with
$2\leq h\leq L$ and a family $\mathcal P$ of at least
$\zeta n^{h-1}$ rainbow $x$--$y$ paths of length $h$.

A fixed vertex is an internal vertex of at most
$O_L(n^{h-2})$ paths in $\mathcal P$: choose its position and the
other $h-2$ internal vertices. The same bound holds for a fixed color
$\gamma$. If $\gamma$ appears on the first or last edge, properness
gives at most one choice for that edge, after which there are
$O_L(n^{h-2})$ choices for the other internal vertices. If $\gamma$
appears at an internal edge position, its color class is a matching
and has at most $n/2$ edges. After choosing that edge, there are
$O_L(n^{h-3})$ choices for the other internal vertices.

Delete from $\mathcal P$ every path having an internal vertex in
$A\setminus\{x,y\}$ or a color in $B$. For some constant $C_L$, the
number of deleted paths is at most
$C_L(|A|+|B|)n^{h-2}\le C_Lrn^{h-1}$. Choose $r_1$ so that
$C_Lr_1\le\zeta/2$. The remaining family $\mathcal P'$ then satisfies
\begin{equation}
 |\mathcal P'|\ge \frac{\zeta}{2}n^{h-1}.
\end{equation}
The endpoints $x$ and $y$ need not lie in $W$, so they are not removed
when they belong to $A$.

For each $P\in\mathcal P'$, let $I_P$ indicate that every internal
vertex of $P$ lies in $W$ and every color of $P$ lies in $R$. Put
$X:=\sum_{P\in\mathcal P'}I_P$ and $\mu:=\E X$. Since $P$ is
simple and rainbow, $\E I_P=\sigma^{h-1}\tau^h$. Hence
\begin{equation}
 \mu=|\mathcal P'|\sigma^{h-1}\tau^h
 \ge \frac{\zeta}{2}\sigma^{L-1}\tau^L n^{h-1}.
\end{equation}
Set $\zeta_*:=\zeta\sigma^{L-1}\tau^L/4$. Then
$\zeta_*n^{h-1}\le\mu/2$.

Two indicators $I_P$ and $I_Q$ are independent unless $P$ and $Q$
share an internal vertex or a color. Sharing $x$ or $y$ does not
matter, since the endpoints are not reserve variables. Each indicator is
dependent on at most $O_L(n^{h-2})$ other indicators. Since
$|\mathcal P'|=O_L(n^{h-1})$, the dependency sum in
\eqref{eq:janson-dependency-sum} is $O_L(n^{2h-3})$. Therefore,
by \eqref{eq:janson-standard},
\begin{equation}
 \Pp\bigl(X<\zeta_*n^{h-1}\bigr)
 \le
 \Pp\bigl(X\le\mu/2\bigr)
 \le
 \exp\left\{-\Omega\left(
 \frac{\mu^2}{\mu+n^{2h-3}}\right)\right\}
 =\exp(-\Omega(n)).
\end{equation}

All constants and failure estimates are uniform over the admissible
choices of $H,A,B,x,y$. There are at most $n^2$ endpoint pairs, $O(n)$
prescribed pairs $(A,B)$, and a bounded number of graphs. The restrictions
of a global reserve pair have exactly the distributions used above, even
when the same color appears in more than one graph. A union bound proves
the simultaneous assertion.
\end{proof}

The following proposition is the main technical result of this
subsection. The graphs $H_j$ are the dense pieces in which the packing
is performed. Each element of $\mathcal T_j$ is one path slot assigned
to $H_j$, and will receive one forest $F_t$. For each
$t\in\mathcal T_j$, the sets $A_t$ and $B_t$ are prescribed forbidden
sets of vertices and colors, while $W_t$ and $R_t$ are private reserves
used later to join the components of $F_t$. The
proposition constructs pairwise edge-disjoint globally rainbow linear
forests that avoid the relevant forbidden and reserve sets, leaving in
each $H_j$ a residual graph whose maximum degree and maximum color
multiplicity are both $O_\ell(rn)+o(n)$.

\begin{proposition}[Rainbow packing with forbidden sets]\label{lem:forbidden-packing}
Fix $d>a>0$, $b,\eta>0$, and choose fixed reserve probabilities
$0<\sigma,\tau<1$ and a fixed integer $\ell\geq3$ so that
$\sigma,\tau,1/\ell$ are sufficiently small in terms of $\eta$.
Let $L=L(d,a,b)$ and $\zeta_*=
\zeta_*(d,a,b,\sigma,\tau)$ be as in
Lemma~\ref{lem:reserved-robust-connections}.
There exist constants $r_0>0$ and $n_2$ such that, for every
$n\ge n_2$ and every $0\le r\le r_0$, the following holds. Let
$H_1,\ldots,H_k$ be vertex-disjoint properly edge-colored graphs, where
\begin{equation}
m_j:=|V(H_j)|,\qquad
\sum_{j=1}^k m_j\leq n,
\qquad m_j\geq dn,\qquad \delta(H_j)\geq dn,
\label{eq:forbidden-host}
\end{equation}
and every cut of $H_j$ with both sides of size at least $an$ has at
least $bn^2$ edges. Let $\mathcal T_1,\ldots,\mathcal T_k$ be pairwise
disjoint index sets, with $\mathcal T_j$ associated with $H_j$, where
\begin{equation}
(1+\eta)\frac{m_j}{2}
 \le |\mathcal T_j|
 \le (1+2\eta)\frac{m_j}{2}.
\label{eq:forbidden-index-count}
\end{equation}
For every $j\in[k]$ and $t\in\mathcal T_j$, prescribe sets $A_t$ of
vertices and $B_t$ of colors such that
\begin{equation}
|A_t|+|B_t|\le rn.
\label{eq:forbidden-sets}
\end{equation}
Assume also that every vertex of $H_j$ belongs to at most $rn$ of the
sets $A_t$ with $t\in\mathcal T_j$, and every color
appearing in $H_j$ belongs to at most $rn$ of the sets $B_t$ with
$t\in\mathcal T_j$.

Then, for every $t\in\bigcup_{j=1}^k\mathcal T_j$, one can choose a
global set $W_t$ of reserved vertices and a global set $R_t$ of reserved
colors, and, for every $j\in[k]$ and $t\in\mathcal T_j$, a globally
rainbow linear forest $F_t$ in $H_j$, such that:

\begin{itemize}
\item every component of $F_t$ is an $\ell$-edge path, and the forests
      $F_t$ are pairwise edge-disjoint;
\item $W_t\cap A_t=\varnothing$, $R_t\cap B_t=\varnothing$, and $F_t$
      avoids $A_t\cup W_t$ and all colors in $B_t\cup R_t$;
\item for every $t\in\bigcup_{j'=1}^k\mathcal T_{j'}$, every $j\in[k]$,
      and every pair of distinct vertices $x,y\in V(H_j)$, there is an integer
      $h$ with $2\leq h\leq L$ such that $H_j$ contains at least
      $\zeta_*n^{h-1}$ rainbow $x$--$y$ paths of length $h$ whose
      internal vertices lie in $W_t$ and whose colors lie in $R_t$;
\item if
      $J_j:=H_j-\bigcup_{t\in\mathcal T_j}E(F_t)$, then
\begin{equation}
\Delta(J_j),\mu_c(J_j)\le Crn+o(n),
\label{eq:forbidden-residual}
\end{equation}
      where $C$ depends only on the fixed constants. The error
      term is uniform for $0\le r\le r_0$.
\end{itemize}
\end{proposition}

\begin{proof}
Choose $r_0\leq r_1(d,a,b,\sigma,\tau)$, decreasing it later if
necessary.
We repeat the construction from Lemmas~\ref{lem:regularisation}
and~\ref{lem:forest-packing}, requiring the paths assigned to $t$ to
avoid $A_t$ and $B_t$. The only new point is to show that these
additional restrictions change each host-edge degree by a relative
$O_\ell(r)$ amount.
After the sets $A_t,B_t$ have been fixed, independently for each $t$,
include every vertex in $\bigcup_jV(H_j)\setminus A_t$ in $W_t$ with
probability $\sigma$, and include every color in
$\bigcup_j\col(H_j)\setminus B_t$ in $R_t$ with probability $\tau$.
All these choices are independent across indices, vertices, and colors.
In particular, $W_t\cap A_t=R_t\cap B_t=\varnothing$.
The simultaneous form of
Lemma~\ref{lem:reserved-robust-connections} gives the third bullet for
all indices, graphs, and endpoint pairs.

Fix $j$ and let $\mathcal P_j$ be the family of unoriented simple rainbow
$\ell$-edge paths in $H_j$. Give these paths the stationary
non-backtracking-walk weights $w(P)$ from Lemma~\ref{lem:stationary}.
Since $m_j\leq n$ and $\delta(H_j)\geq dn$, we have
$\delta(H_j)\geq dm_j$. Applying \eqref{eq:path-weight-bound} to
$H_j$, and then using $m_j\geq dn$, gives a constant
$C_0=C_0(d,\ell)$ such that $w(P)\leq C_0n^{-\ell-1}$ uniformly in
$j$. Fix $0<a_0\leq C_0^{-1}$. Then
$a_0n^{\ell+1}w(P)\leq1$ for every $P\in\mathcal P_j$. For every admissible
pair $(t,P)$, where $t\in\mathcal T_j$ and $P$ avoids
\begin{equation}
A_t\cup W_t\quad\text{and}\quad B_t\cup R_t,
\end{equation}
form the same $(3\ell+1)$-uniform auxiliary hyperedge as in
equation~\eqref{eq:auxiliary-edge}, using the vertices $E_e$,
$V_{t,v}$, and $C_{t,\gamma}$. Sample it independently with probability
$a_0n^{\ell+1}w(P)$, and let $\mathcal H_j$ be the resulting auxiliary
hypergraph. Put
\begin{equation}
\kappa:=(1-\sigma)^{\ell+1}(1-\tau)^\ell,
\qquad
D_{0,j}:=|\mathcal T_j|\kappa a_0n^{\ell+1}
              \frac{\ell}{e(H_j)}=\Theta(n^\ell).
\label{eq:forbidden-D0}
\end{equation}

There is a constant $C_1$ depending only on the fixed parameters such
that the stationary-walk estimates imply, uniformly in $e$,
\begin{equation}
d_{\mathcal H_j}(E_e)=\bigl(1\pm C_1r\pm o(1)\bigr)D_{0,j}.
\label{eq:forbidden-edge-degree}
\end{equation}
We give the details. For $e\in E(H_j)$ and $t\in\mathcal T_j$, let
$Y_{t,e}$ be the $w$-mass of the paths containing $e$ that avoid
$A_t\cup W_t$ and all colors in $B_t\cup R_t$. The endpoints and the
color of $e$ exclude at most $3rn$ indices, by the incidence assumptions
on $A_t$ and $B_t$. Put
\begin{equation}
 M_e:=\sum_{\substack{P\in\mathcal P_j\\e\in E(P)}}w(P)
 =(1+o(1))\frac{\ell}{e(H_j)}=\Theta(n^{-2}).
\end{equation}
For every other index, expose the walk in both directions from the
occurrence of $e$. At each remaining position, the conditional
probability of a fixed vertex outside $V(e)$ or a fixed color different
from $c(e)$ is $O_\ell(n^{-1})$.
Thus the $w$-mass of the paths containing $e$ and meeting any fixed
additional vertex or color is at most $C_\ell M_e/n$. A union bound
over $A_t$ and $B_t$ shows that the mass deleted by the prescribed
forbidden sets is at most $C_\ell rM_e$. Every remaining path survives
the random reserves with probability
$\kappa=(1-\sigma)^{\ell+1}(1-\tau)^\ell$. Consequently,
\eqref{eq:forbidden-sets} and the bound on the excluded indices give
\begin{equation}
 \E\sum_{t\in\mathcal T_j}Y_{t,e}
 =\bigl(1\pm O_\ell(r)\pm o(1)\bigr)
   |\mathcal T_j|\kappa\frac{\ell}{e(H_j)}
 =\Theta(n^{-1}).
\end{equation}
The variables $Y_{t,e}$ are independent over $t$, and
$0\leq Y_{t,e}\leq M_e=O(n^{-2})$. Since
$|\mathcal T_j|=O(n)$, the sum of the squared lengths of their ranges
is $O(n^{-3})$. Taking $\rho=n^{-1/10}$ and applying
\eqref{eq:hoeffding-standard} with
$s=\rho\E\sum_tY_{t,e}=\Theta(\rho n^{-1})$ gives
\begin{equation}
 \Pp\left(\left|\sum_{t\in\mathcal T_j}Y_{t,e}
 -\E\sum_{t\in\mathcal T_j}Y_{t,e}\right|>
 \rho\E\sum_{t\in\mathcal T_j}Y_{t,e}\right)
 \leq \exp(-\Omega(n^{4/5})).
\end{equation}
A union bound shows that this reserve estimate holds simultaneously for
every host edge in every $H_j$, since $k\leq1/d$ and there are
$O(n^2)$ host edges. Conditional on reserve sets satisfying it,
$d_{\mathcal H_j}(E_e)$ is a sum of independent Bernoulli variables
with mean $a_0n^{\ell+1}\sum_tY_{t,e}=\Theta(n^\ell)$.
Equation~\eqref{eq:chernoff-standard}, again with relative error $\rho$,
gives failure probability $\exp(-\Omega(n^{\ell-1/5}))$. A union bound
over the $O(n^2)$ host edges proves
\eqref{eq:forbidden-edge-degree} uniformly.

If $v\in A_t\cup W_t$ or $\gamma\in B_t\cup R_t$, the corresponding
indexed degree is zero. To bound all other indexed degrees, as required in
\eqref{eq:forbidden-indexed-degrees},
couple this construction with the corresponding construction in which
$A_t$ and $B_t$ are absent. Give the vertices in $A_t$ and the colors
in $B_t$ independent auxiliary reserve indicators,
with probabilities $\sigma$ and $\tau$, respectively, and use the same
indicators for all other vertices and colors and the same auxiliary-edge
sampling variables in both constructions. Every path admissible in the
present construction avoids $A_t$ and $B_t$, so it is also admissible in the
comparison construction. Hence, under this coupling, the forbidden
sets can only reduce the degrees of the vertices $V_{t,v}$ and
$C_{t,\gamma}$.
Repeat the proof of Lemma~\ref{lem:regularisation} for the comparison
construction, using ambient scale $n$, unreserved probabilities
$1-\sigma$ and $1-\tau$, and index set $\mathcal T_j$. Its estimates
hold simultaneously with probability $1-o(1)$. The bounds
\eqref{eq:degree-vertex}
and~\eqref{eq:degree-color},
together with \eqref{eq:forbidden-index-count}, give
\begin{equation}
\frac{d_{\mathcal H_j}(V_{t,v})}{D_{0,j}}
 \le \frac{1+1/\ell+o(1)}{(1+\eta)(1-\sigma)},
\qquad
\frac{d_{\mathcal H_j}(C_{t,\gamma})}{D_{0,j}}
 \le \frac{1+o(1)}{(1+\eta)(1-\tau)}.
\label{eq:forbidden-indexed-degrees}
\end{equation}
The second estimate is needed only when the color class has size at
least $n^{1/2}$; every smaller color class has indexed degree
$o(D_{0,j})$ by the corresponding part of
Lemma~\ref{lem:regularisation}.
By the choice of $\sigma,\tau$ and $\ell$ above, both quantities in
\eqref{eq:forbidden-indexed-degrees} are at most $1-\xi$ for some fixed
$\xi>0$. Under the same coupling, the present auxiliary hypergraph is
a subhypergraph of the comparison hypergraph. Hence
\eqref{eq:auxiliary-codegree} gives
\begin{equation}
\Delta_2(\mathcal H_j)=O(D_{0,j}/n)=O(n^{\ell-1}).
\label{eq:forbidden-codegree}
\end{equation}
The reserve-connection event and all the degree, codegree, and comparison
events above hold simultaneously with probability $1-o(1)$. Fix a
realization for which they all hold.

Choose a positive sequence $\rho_n\to0$ which uniformly dominates all
the $o(1)$ terms above, and choose a fixed constant
$C_2>C_1+1$. Put
\begin{equation}
D_j=(1+C_2r+\rho_n)D_{0,j}.
\end{equation}
After reducing the fixed admissible upper bound $r_0$, if necessary,
we have $\Delta(\mathcal H_j)\le D_j$. With
$\beta=1/(4\ell)$, \eqref{eq:forbidden-codegree} also gives
$\Delta_2(\mathcal H_j)\le D_j^{1-\beta}$ for all sufficiently large
$n$. Thus the degree and codegree hypotheses of the
Ehard--Glock--Joos theorem hold with parameter $D_j$.

Use the vertex weights
\begin{equation}
\omega_v(t,P):=d_P(v)
\end{equation}
and, for every color satisfying
$|E_\gamma\cap E(H_j)|\ge n^{1/2}$, the color
weights
\begin{equation}
\omega_\gamma(t,P):=\mathbf1_{\{\gamma\in\col(P)\}}.
\end{equation}
Use $\beta=1/(4\ell)$ in Theorem~\ref{thm:egj}. Since
$D_j=\Theta(n^\ell)$, every vertex weight has total mass
$\Omega(nD_{0,j})$, while every color weight under consideration has
total mass $\Omega(n^{1/2}D_{0,j})$; the maximum weight of one auxiliary
edge is at most two. Hence all these weights satisfy
condition~\eqref{eq:egj-weight-condition}.
The auxiliary edge set and the weight family have polynomial size, so
the remaining size hypotheses of Theorem~\ref{thm:egj} also hold.
Theorem~\ref{thm:egj}
therefore applies with parameter $D_j$ and this weight family.
Every sampled path contributes once to $\omega_v$ for each of its edges
incident with $v$. Since every sampled path is rainbow, it contributes
once to $\omega_\gamma$ precisely when it contains an edge of color
$\gamma$. Therefore
\begin{equation}
 \omega_v(E(\mathcal H_j))
  =\sum_{e\in E_{H_j}(v)}d_{\mathcal H_j}(E_e),
 \qquad
 \omega_\gamma(E(\mathcal H_j))
  =\sum_{e\in E_\gamma\cap E(H_j)}d_{\mathcal H_j}(E_e).
\end{equation}
Equation~\eqref{eq:forbidden-edge-degree} now gives
\begin{equation}
\omega_v(E(\mathcal H_j))
 =\bigl(1\pm O_\ell(r)\pm o(1)\bigr)d_{H_j}(v)D_{0,j},
\end{equation}
and
\begin{equation}
\omega_\gamma(E(\mathcal H_j))
 =\bigl(1\pm O_\ell(r)\pm o(1)\bigr)
     |E_\gamma\cap E(H_j)|D_{0,j}.
\end{equation}
Let $\mathcal M_j$ be the matching supplied by
Theorem~\ref{thm:egj}. Since
$D_j=(1+O_\ell(r)+o(1))D_{0,j}$, its weighted conclusion gives
\begin{equation}
 \omega_v(\mathcal M_j)
  \geq d_{H_j}(v)-O_\ell(rn)-o(n),
 \qquad
 \omega_\gamma(\mathcal M_j)
  \geq |E_\gamma\cap E(H_j)|-O_\ell(rn)-o(n).
\end{equation}
The second estimate is needed only for color classes of size at least
$n^{1/2}$; every smaller class already has size $o(n)$. Thus at most
$O_\ell(rn)+o(n)$ edges remain at every vertex and in every color
class. Since the chosen
hyperedges form a matching, no two contain the same $E_e$; hence all
selected paths are edge-disjoint. For each $t$, no two contain
the same $V_{t,v}$ or $C_{t,\gamma}$; hence their union is a globally
rainbow linear forest. This proves \eqref{eq:forbidden-residual} and
completes the proof of the proposition.
\end{proof}

\subsection{Arbitrary dense graphs}\label{subsec:arbitrary-dense}

We now prove part~(i). The first lemma extends a globally rainbow linear
forest in a graph of linear minimum degree. The next two lemmas prepare the
application of Proposition~\ref{lem:forbidden-packing} to an arbitrary
dense graph.

\begin{lemma}[Linear-forest extension]\label{lem:forest-extension}
Fix $\alpha>0$ and an integer $\ell\geq1$. Let $(G,c)$ be a properly
edge-colored graph on $n$ vertices with minimum degree
$\delta\geq\alpha n$. If $F$ is a globally rainbow linear forest whose
components are $\ell$-edge paths, then $E(F)$ can be covered by at most
\begin{equation}
 (12+\frac{12}{\ell}+o(1))
 \frac{n\,e(F)}{\delta^2}
 +(2+o(1))\frac n\delta
\end{equation}
rainbow paths in $G$. The error is uniform over $F$.
\end{lemma}

\begin{proof}
Let $s=e(F)/\ell$ be the number of components of $F$. The assertion is
immediate when $s=0$. Put
\begin{equation}
 L=\lfloor\frac{\delta}{6(\ell+1)}\rfloor,
 \qquad B=\delta-2(\ell+2)L+4,
\end{equation}
and partition the components of $F$ into
$r=\lceil s/L\rceil$ groups, each containing at most $L$ components.
We show that every group can be covered by at most $n/B$ rainbow paths.

Fix one group $\mathcal T$ of $t\leq L$ target paths, and let
$\mathcal C$ be the set of their $\ell t$ colors. During the
construction, let $\mathcal R$ be the set of unprocessed target paths
and let $\mathcal K$ be the connector colors already used for this
group; initially, $\mathcal R=\mathcal T$ and $\mathcal K=\varnothing$.
Choose one target path, orient it arbitrarily, call it $P$, and remove
it from $\mathcal R$. If the current exit of $P$ is $x$, extend it
whenever one of the following operations is possible.

\begin{enumerate}
\item Choose $Q\in\mathcal R$, orient it from $y$ to $z$, and append
the edge $xy$ followed by $Q$, provided that $xy\in E(G)$ and
$c(xy)\notin\mathcal C\cup\mathcal K$.
\item Choose $Q\in\mathcal R$, orient it from $y$ to $z$, and choose
$u\notin V(P)\cup V(\mathcal R)$ such that $xu,uy\in E(G)$ and both
connector colors lie outside $\mathcal C\cup\mathcal K$. Append
$xuy$ followed by $Q$.
\end{enumerate}

In either operation, remove $Q$ from $\mathcal R$ and add the connector
colors to $\mathcal K$. The target paths are vertex-disjoint and
globally rainbow. The connector vertices are new, the connector
colors are new, and in the second operation the two connector colors
are distinct by properness at $u$. Thus the path remains simple and
rainbow.

When neither operation is possible, close the current path, call it
$P_i$, and denote its exit by $x_i$. Let $\mathcal R_i$ and
$\mathcal K_i$ be the sets present at that time, and define
\begin{equation}
 U_i=\{u\notin V(P_i)\cup V(\mathcal R_i):
 x_iu\in E(G),\
 c(x_iu)\notin\mathcal C\cup\mathcal K_i\}.
\end{equation}
If $\mathcal R_i$ is non-empty, choose and orient a path
$Q\in\mathcal R_i$, set $P=Q$ and
$\mathcal R=\mathcal R_i\setminus\{Q\}$, and repeat. Continue until
every target path in $\mathcal T$ has been processed.

At most $2(t-1)$ connector colors are used throughout this construction, so
$|\mathcal C\cup\mathcal K_i|\leq(\ell+2)t-2$. The set
$V(P_i)\cup V(\mathcal R_i)$ contains at most the
$(\ell+1)t$ target vertices and $t-1$ internal connector vertices.
After omitting $x_i$, it therefore excludes at most
$(\ell+2)t-2$ neighbors of $x_i$. Properness allows each forbidden
color to exclude at most one further neighbor. Hence
\begin{equation}\label{eq:available-set}
 |U_i|\geq\delta-2(\ell+2)t+4\geq B\geq\frac{\delta}{2}.
\end{equation}

The sets $U_i$ are pairwise disjoint. Suppose that $i<j$ and
$u\in U_i\cap U_j$. Let $Q$ be the last target path traversed by
$P_j$, and let $x_j$ be the exit of $P_j$. When $P_i$ was closed,
$Q$ still belonged to $\mathcal R_i$, so $u\notin V(Q)$. Moreover,
$\mathcal K_i\subseteq\mathcal K_j$. The edges $x_iu$ and $ux_j$
therefore have colors outside
$\mathcal C\cup\mathcal K_i$, and their colors are distinct because
they meet at $u$. Appending $x_iux_j$ and then traversing $Q$ in
reverse would have been a legal extension of the second type. This is
a contradiction.

If $\mathcal T$ produces $p$ paths, \eqref{eq:available-set} and the
disjointness of the sets $U_i$ give $pB\leq n$. Consequently, all
$r$ groups are covered by at most
\begin{equation}
 \frac nB(\frac sL+1)
 \leq
 (12+\frac{12}{\ell}+o(1))
 \frac{n\,e(F)}{\delta^2}
 +(2+o(1))\frac n\delta
\end{equation}
paths, because $B\geq\delta/2$ and
$L=(1+o(1))\delta/(6(\ell+1))$. This proves the lemma.
\end{proof}

\begin{lemma}[Nested cut decomposition]\label{lem:nested-robust-decomposition}
Fix $0<a<1$ and then an integer $K>1/a$. Choose $\theta_K>0$. For
$i=K-1,\ldots,0$, choose $\theta_i>0$ sufficiently small in terms of
$\theta_{i+1}$. Thus
\begin{equation}
 0<\theta_0\ll\theta_1\ll\cdots\ll\theta_K.
\end{equation}
Every graph $G$ on $n$ vertices admits an index $i<K$ and a partition
$\mathcal P$ into at most $K$ parts, each of size at least $an$. If $C$
is the graph formed by the edges joining distinct parts, then
\begin{equation}\label{eq:few-cross-edges}
 e(C)<K\theta_i n^2.
\end{equation}
Moreover, for every $U\in\mathcal P$ and every partition
$U=A\mathbin{\dot\cup}B$ with $|A|,|B|\ge an$,
\begin{equation}\label{eq:robust-parts}
 e_G(A,B)\ge\theta_{i+1}n^2.
\end{equation}
\end{lemma}

\begin{proof}
Set $\mathcal P_{-1}:=\{V(G)\}$. For $t=0,1,\ldots,K$, obtain
$\mathcal P_t$ from $\mathcal P_{t-1}$ by repeatedly splitting a
current part $U$ as
$U=A\mathbin{\dot\cup}B$ whenever
\begin{equation}
|A|,|B|\ge an
\quad\text{and}\quad
e_G(A,B)<\theta_tn^2.
\end{equation}
Every part created by this process has at least $an$ vertices. Hence
each partition has at most $1/a$ parts. Since each split increases the
number of parts by one and $K>1/a$, fewer than $K$ splits occur in
total. Therefore, there is an $i<K$ such that
\begin{equation}\label{eq:stable-partition}
\mathcal P_i=\mathcal P_{i+1}.
\end{equation}
Let $C$ be the graph formed by the edges joining distinct parts of this
common partition. Every edge of $C$ is counted at the unique split that
first separates its endpoints. Each such split occurred at some level
$t\leq i$ and cut fewer than $\theta_tn^2\leq\theta_in^2$ edges. Since
there were fewer than $K$ splits,
\begin{equation}
e(C)<K\theta_i n^2.
\end{equation}
On the other hand, \eqref{eq:stable-partition} and terminality at level
$i+1$ imply that every part $U\in\mathcal P_i$ satisfies
\begin{equation}
e_G(A,B)\ge\theta_{i+1}n^2
\end{equation}
for every partition $U=A\mathbin{\dot\cup}B$ with
$|A|,|B|\ge an$.
\end{proof}

Given a graph $G$, a set $X\subseteq V(G)$, a set
$\Gamma\subseteq\col(G)$, and a partition
$V(G)\setminus X=V_1\mathbin{\dot\cup}\cdots\mathbin{\dot\cup}V_k$,
\textbf{an ordered list of pieces} is a list
$\mathcal Q=(P_1,\ldots,P_r)$ of $r\geq0$ paths, each with a specified
first and last vertex, satisfying the following conditions. The paths
$P_1,\ldots,P_r$ are pairwise vertex-disjoint, and their sets of colors
are pairwise disjoint. Each $P_i$ is either a two-edge path $uxv$ with
$x\in X$ and $u,v\in V(G)\setminus X$, or a single edge of $G-X$ whose
color lies in $\Gamma$. For every $i\in\{1,\ldots,r-1\}$, there exists
$j_i\in\{1,\ldots,k\}$ such that the last vertex of $P_i$ and the first
vertex of $P_{i+1}$ both lie in $V_{j_i}$. The indices $j_i$ need not be
distinct.

\begin{lemma}[Organizing exceptional edges]\label{lem:exceptional-sequences}
Let $X\subseteq V(G)$ and $\Gamma\subseteq\col(G)$. Suppose that
$V(G)\setminus X=V_1\mathbin{\dot\cup}\cdots\mathbin{\dot\cup}V_k$.
Fix $\eta>0$. For each $j\in[k]$, set $m_j:=|V_j|$, and put
\begin{equation}
s:=|X|+|\Gamma|,\qquad
q_j:=\lceil(1+\eta)\frac{m_j}{2}\rceil+20s.
\label{eq:sequence-index-count}
\end{equation}
Let $\mathcal T_1,\ldots,\mathcal T_k$ be pairwise disjoint index sets
with $|\mathcal T_j|=q_j$. For every
$t\in\bigcup_{j=1}^k\mathcal T_j$, there is an ordered list of pieces
$\mathcal Q_t$. Together, these lists cover every edge in
\begin{equation}
E_G(X,V(G)\setminus X)
 \cup\bigcup_{\gamma\in\Gamma}E_\gamma(G-X),
\label{eq:exceptional-interface}
\end{equation}
except possibly some edges forming a graph $L$. This graph satisfies
\begin{equation}
\Delta(L),\mu_c(L)\le |X|.
\label{eq:exceptional-leftover}
\end{equation}
If $t\in\mathcal T_j$ and $\mathcal Q_t$ is non-empty, its first piece
starts in $V_j$. Each list contains at most $s$ pieces. Every vertex in
$V(G)\setminus X$ occurs in at most $s$ lists. Every color outside
$\Gamma$ occurs in at most $|X|$ lists.
\end{lemma}

\begin{proof}
Process first the vertices of $X$ and then the colors of
$\Gamma$, one at a time.  At the stage of $x\in X$, if
$d_G(x,V(G)\setminus X)$ is odd, put one incident edge into $L$; pair
the remaining star edges to form two-edge pieces $u x v$.  At the stage
of $\gamma\in\Gamma$, use as pieces all edges of the matching
$E_\gamma(G-X)$.  Replace every piece by an edge of the quotient
multigraph on $[k]$, allowing loops.  Orient its non-loop edges so that
the indegree and outdegree at each quotient vertex differ by at most
one; orient each loop arbitrarily as an actual path, but let it stay at
its quotient vertex.  If the quotient degree at $j$ is $d_j$, the number
of loops plus edges directed out of $j$ is at most
\begin{equation}
\lceil\frac{d_j}{2}\rceil
 \le\lceil\frac{m_j}{2}\rceil.
\label{eq:quotient-outdegree}
\end{equation}

At the beginning of a stage, an index $t$ currently at a quotient vertex
is admissible for an oriented piece if that piece repeats no actual
vertex or color already used in $\mathcal Q_t$. After any number of stages,
the number of indices currently in $V_j$ differs from $q_j$ by at most
the number of completed stages, and hence by at most $s$.  A vertex of
$V(G)\setminus X$ can have occurred in at most one earlier edge at each
vertex of $X$ and in at most one earlier edge of each color in $\Gamma$.
A color outside $\Gamma$ can have occurred on at most one edge at each
vertex of $X$.  The same is true of a color in $\Gamma$ during all
$X$-stages; hence, at its own color stage, it has occurred previously
on at most $|X|$ edges incident with $X$. It follows that each piece is
inadmissible for at most $4s$ current indices. By
\eqref{eq:sequence-index-count}, \eqref{eq:quotient-outdegree}, and the
extra $20s$ indices, every piece at $j$ has a list of at least as many
admissible indices as the total number of pieces directed out of or
looping at $j$.  Hall's theorem gives an injective assignment of the
pieces to indices. Move each assigned index along its oriented piece.
The balanced orientation changes the number of indices at any quotient
vertex by at most one, closing the induction.

No index receives two pieces in one stage. Thus the central vertices
$x$, and later the colors $\gamma$, do not repeat in any $\mathcal Q_t$;
the lists prevent all other repetitions. This proves the assertions about
the lists. The discarded graph $L$ has at most one edge at each
$x\in X$, so \eqref{eq:exceptional-leftover} follows
from properness.
\end{proof}

We are now ready to present the proof of
Theorem~\ref{thm:linear-minimum-degree}(i). We first explain its main
steps.

Lemma~\ref{lem:nested-robust-decomposition} partitions $V(G)$ into a
bounded number of large parts. Each part has many edges across every large
cut, while few edges join different parts. The cross edges may, however, be
concentrated at a few vertices or in a few colors. Let $X$ be the set of
vertices with large cross-degree. After deleting $X$, let $\Gamma$ be the
set of colors with large multiplicity among the remaining cross edges.
Inside each part, deleting $X$ and all edges with colors in $\Gamma$ leaves
a dense robust graph $H_j$. The other cross edges have small maximum degree
and small color multiplicity.

The exceptional edges must still be covered.
Lemma~\ref{lem:exceptional-sequences} places the edges with exactly one
endpoint in $X$, and the $\Gamma$-colored edges outside $X$, into ordered
lists $\mathcal Q_t$, apart from a small leftover. For every
$t\in\mathcal T_j$, Proposition~\ref{lem:forbidden-packing} constructs a
globally rainbow linear forest $F_t\subseteq H_j$. It also gives reserves
$W_t$ and $R_t$ assigned to $t$. The forest avoids all vertices and colors
used by $\mathcal Q_t$. The order of $\mathcal Q_t$ ensures that each
required join can be made inside one of the dense graphs. Using $W_t$ and
$R_t$, we join the components of $F_t$ and the pieces of $\mathcal Q_t$ into
one rainbow path. Thus the exceptional edges placed in the lists do not
require additional main paths. Here $X$ and $\Gamma$ isolate the
concentrations; $W_t$ and $R_t$ are the actual connector reserves.

The edges still uncovered form a graph of small maximum degree and small
color multiplicity. They lie among the other cross edges, the small
leftover above, the edges inside $X$, and the internal residual graphs left
by the packing. We split this graph into globally rainbow matchings and
apply Lemma~\ref{lem:forest-extension} to each matching. The main family
uses at most $n/2+\varepsilon n/4$ paths, and the residual graph uses at most
$\varepsilon n/4$ further paths.

\begin{proof}[Proof of Theorem~\ref{thm:linear-minimum-degree}(i)]
Assume that $0<\varepsilon\le1$. We first fix the constants. Set
\begin{equation}
a:=\frac{\alpha}{16},
\qquad
K:=\lceil\frac1a\rceil+2.
\end{equation}
Define
\begin{equation}
A_\alpha:=\frac{12}{\alpha^2}+\frac6\alpha+1,
\qquad
B_\alpha:=2+5K.
\end{equation}
These constants collect the losses in the final path counts. Choose
\begin{equation}
0<\sigma,\tau\ll\eta\ll\varepsilon,\alpha.
\end{equation}
We also require $\eta\le\varepsilon/8$. Choose
\begin{equation}
0<\lambda\ll\varepsilon,\alpha.
\end{equation}
Decrease $\lambda$, if necessary, so that $\lambda\le\alpha/4$ and
\begin{equation}
A_\alpha\lambda\le\frac{\varepsilon}{16}.
\end{equation}
We choose the thresholds backward. First choose
$0<\theta_K\ll\lambda$. For $i=K-1,\ldots,0$, suppose that
$\theta_{i+1}$ has already been chosen. Apply
Lemma~\ref{lem:reserved-robust-connections} with
\begin{equation}
d=\frac\alpha2,
\qquad b=\frac{\theta_{i+1}}2
\end{equation}
and the fixed reserve probabilities. Let $L_i$ be the resulting length
bound, and write $\zeta_i$ for the constant $\zeta_*$ in
\eqref{eq:reserved-connections}. Choose a fixed integer $\ell_i$ so large
that
\begin{equation}
\frac1{\ell_i}\ll\zeta_i
\label{eq:forest-length-choice}
\end{equation}
and both inequalities in \eqref{eq:forbidden-indexed-degrees} have fixed
positive slack. Let $C_i\geq1$ be a constant for which
\eqref{eq:forbidden-residual} holds with this choice of $\ell_i$. Finally,
let $r_i^0>0$ be a common admissible upper bound for $r$ in
Lemma~\ref{lem:reserved-robust-connections} and
Proposition~\ref{lem:forbidden-packing}.

Now choose $0<\theta_i<\theta_{i+1}$ so small that, with
\begin{equation}
r_i:=\frac{3K\theta_i}{\lambda},
\label{eq:exceptional-rate}
\end{equation}
we have
\begin{equation}
r_i\ll
\min\{\theta_{i+1},\eta,\zeta_i,
            \frac{\varepsilon}{A_\alpha B_\alpha C_i},
            \frac{\varepsilon}{K}\},
\qquad
\frac1{\ell_i}+r_i\ll\zeta_i.
\label{eq:exceptional-rate-choice}
\end{equation}
We also impose
\begin{equation}
r_i\le
\min\{
 \frac{\eta\alpha}{200},
 \frac{\varepsilon}{160K},
 \frac{\varepsilon}{16A_\alpha B_\alpha C_i},
 \frac{r_i^0}{5},
 \frac{\theta_{i+1}}2,
 \frac\alpha4
\}.
\label{eq:exceptional-rate-explicit}
\end{equation}
This recursion is well-defined: $L_i,\zeta_i,\ell_i,C_i$, and $r_i^0$
depend only on constants fixed before $\theta_i$ is chosen. Increase
$n_0=n_0(\alpha,\varepsilon)$, if necessary, so that $n$ exceeds the
finitely many lower thresholds in
Lemma~\ref{lem:reserved-robust-connections} and
Proposition~\ref{lem:forbidden-packing} for $i=0,\ldots,K-1$. We also
require $n$ to exceed the thresholds in all subsequent asymptotic
estimates.

We now decompose $G$. Apply
Lemma~\ref{lem:nested-robust-decomposition}. Write
$\mathcal P=\{U_1,\ldots,U_k\}$ for the resulting partition, let $i<K$
be its level, and let $C$ be its cross-edge graph. Let $X$ contain the
vertices of large cross-degree:
\begin{equation}
X:=\{v:d_C(v)>\lambda n\}.
\end{equation}
In $C-X$, let $\Gamma$ contain the colors of large multiplicity:
\begin{equation}
\Gamma:=\{\gamma:|E_\gamma\cap E(C-X)|>\lambda n\}.
\end{equation}
By \eqref{eq:few-cross-edges},
\begin{equation}
\begin{aligned}
 |X|\lambda n&\leq 2e(C)<2K\theta_i n^2,\\
 |\Gamma|\lambda n&\leq e(C)<K\theta_i n^2.
\end{aligned}
\end{equation}
Thus, for $s:=|X|+|\Gamma|$,
\begin{equation}
s<\frac{3K\theta_i}{\lambda}n=r_i n.
\label{eq:exceptional-size}
\end{equation}

Since $|U_j|\ge an$ and $|X|<r_i n<an$, every set
$V_j:=U_j\setminus X$ is non-empty. Put
\begin{equation}
H_j:=G[V_j]-\bigcup_{\gamma\in\Gamma}E_\gamma,
\qquad m_j:=|V_j|.
\end{equation}
Fix $v\in V_j$. Deleting the cross edges costs at most $\lambda n$
edges at $v$, and deleting $X$ costs at most $|X|$ more. By properness,
deleting the colors in $\Gamma$ costs at most $|\Gamma|$ edges. Hence,
by \eqref{eq:exceptional-size} and
\eqref{eq:exceptional-rate-explicit},
\begin{equation}
d_{H_j}(v)
 \ge\alpha n-\lambda n-|X|-|\Gamma|
 \ge\frac\alpha2n.
\label{eq:internal-minimum-degree}
\end{equation}
In particular, $m_j\ge\alpha n/2$. Now suppose that
$V_j=A\mathbin{\dot\cup}B$ with $|A|,|B|\ge an$. Apply
\eqref{eq:robust-parts} inside $U_j$ to the cut
$A\mathbin{\dot\cup}(B\cup(X\cap U_j))$. Deleting the edges from $A$
to $X$, and then all edges with colors in $\Gamma$, gives
\begin{equation}
e_{H_j}(A,B)
 \ge\theta_{i+1}n^2-|X|n-|\Gamma|n
 \ge\frac{\theta_{i+1}}2n^2.
\label{eq:internal-robustness}
\end{equation}
Thus every $H_j$ has the minimum-degree and robust-cut properties used
to define $L_i$ and $\zeta_i$.

Let $R_{\mathrm{cross}}$ consist of the cross edges of $C-X$ whose
colors do not belong to $\Gamma$. By definition,
\begin{equation}
\Delta(R_{\mathrm{cross}}),
\mu_c(R_{\mathrm{cross}})\le\lambda n.
\label{eq:cross-residual}
\end{equation}

We next organize the edges removed before the packing.
Apply Lemma~\ref{lem:exceptional-sequences} to the sets $V_j$, the
exceptional vertex set $X$, and the exceptional color set $\Gamma$.
The resulting lists $\mathcal Q_t$ contain every edge with exactly one
endpoint in $X$, except for the edges of the graph $L$ in
\eqref{eq:exceptional-leftover}. They also contain every
$\Gamma$-colored edge in $G-X$.

For every $t$, let $A_t:=V(\mathcal Q_t)$ and
$B_t:=\operatorname{col}(\mathcal Q_t)$. Since $\mathcal Q_t$ has at
most $s$ pieces, $|A_t|\le3s$ and $|B_t|\le2s$. For fixed $j$, each
vertex of $H_j$ and each color appearing in $H_j$ occurs in at most
$s$ of the corresponding forbidden sets. The only possible larger
multiplicities come from $X$ and $\Gamma$, which do not occur in
$H_j$. Thus the forbidden-set assumptions in
Proposition~\ref{lem:forbidden-packing} hold with
$r=5s/n<5r_i\le r_i^0$.

By \eqref{eq:sequence-index-count},
$q_j\ge(1+\eta)m_j/2$. Moreover,
$q_j\le(1+\eta)m_j/2+1+20s$. The bounds
$m_j\ge\alpha n/2$, $s<r_i n$, and
\eqref{eq:exceptional-rate-explicit} give
$1+20s\le\eta m_j/2$ for sufficiently large $n$. Hence
\begin{equation}
(1+\eta)\frac{m_j}{2}
 \le q_j
 \le(1+2\eta)\frac{m_j}{2}.
\end{equation}
We may therefore apply Proposition~\ref{lem:forbidden-packing} to
$H_1,\ldots,H_k$ with $\ell=\ell_i$. It supplies a linear forest $F_t$
for every $t$, reserve sets $W_t,R_t$, and internal residual graphs
$J_j$ satisfying \eqref{eq:forbidden-residual}.

We now join the forests to their exceptional lists. Fix
$t\in\mathcal T_j$. Form an ordered sequence $\mathcal S_t$ by first
listing and orienting the components of $F_t$ arbitrarily, and then
appending the pieces of $\mathcal Q_t$ in their prescribed order.
Consecutive components of $F_t$ have their relevant endpoints in
$V_j$. If both $F_t$ and $\mathcal Q_t$ are non-empty, the first piece
of $\mathcal Q_t$ starts in $V_j$, so the join between them can also be
made there. Finally, for two consecutive pieces of $\mathcal Q_t$, the
relevant endpoints lie in a common set $V_{j'}$ by the definition of an
ordered list. Thus every required join is between two distinct vertices
of some $H_{j'}$.

The paths in $\mathcal S_t$ are pairwise vertex-disjoint and together
are rainbow. Indeed, these properties hold separately for $F_t$ and
$\mathcal Q_t$, while $F_t$ avoids
$A_t=V(\mathcal Q_t)$ and $B_t=\operatorname{col}(\mathcal Q_t)$.
Moreover, $W_t\cap V(\mathcal Q_t)=\varnothing$ and
$R_t\cap\operatorname{col}(\mathcal Q_t)=\varnothing$, while $F_t$
avoids $W_t$ and all colors in $R_t$. We use internal vertices from
$W_t$ and colors from $R_t$ to make the joins. Since $F_t$ has at most
$n/(\ell_i+1)$ components and $\mathcal Q_t$ has at most $s$ pieces,
the number of joins is at most
\begin{equation}
\frac{n}{\ell_i+1}+s.
\label{eq:number-of-joins}
\end{equation}
For each join, \eqref{eq:reserved-connections} gives at least
$\zeta_i n^{h-1}$ candidate connectors of some length $h\le L_i$. A
reserve vertex or color used earlier for this index belongs to only
$O_{L_i}(n^{h-2})$ candidates. By
\eqref{eq:exceptional-rate-choice}, \eqref{eq:exceptional-size}, and
\eqref{eq:number-of-joins}, all previously used reserve vertices and
colors together exclude at most
\begin{equation}
O_{L_i}\left((\frac1{\ell_i}+r_i)n^{h-1}\right)
 <\zeta_i n^{h-1}
\end{equation}
candidates. We can therefore choose the connectors greedily. Each new
connector avoids every reserve vertex and color used earlier for this
index. Its internal vertices and colors also avoid all paths in
$\mathcal S_t$. Adding the connectors in order produces one simple
rainbow path containing every path in $\mathcal S_t$. If
$\mathcal S_t$ is empty, omit $t$; if it has one member, no connector is
needed. Hence all edges in the forests $F_t$ and all edges contained in
the lists $\mathcal Q_t$ are covered by at most
\begin{equation}
q:=\sum_{j=1}^kq_j
\label{eq:forest-exceptional-path-count}
\end{equation}
rainbow paths.

We now cover the edges not yet covered. Every such edge belongs to
\begin{equation}
R^*:=R_{\mathrm{cross}}\cup L\cup G[X]
                  \cup\bigcup_{j=1}^kJ_j
\end{equation}
and
\begin{equation}
D:=\max\{\Delta(R^*),\mu_c(R^*)\}.
\end{equation}
The union $L\cup G[X]$ has maximum degree and color multiplicity at
most $2|X|$. Each vertex outside $X$ lies in only one $J_j$, while a
color may occur in at most $K$ of the graphs $J_j$. Since the packing
proposition was applied with $r=5s/n$, equations
\eqref{eq:forbidden-residual}, \eqref{eq:exceptional-size}, and
\eqref{eq:cross-residual} give
\begin{equation}
D\le\lambda n+B_\alpha C_i s+o(n)
 \le\bigl(\lambda+B_\alpha C_i r_i+o(1)\bigr)n.
\label{eq:final-residual-degree}
\end{equation}

By Lemma~\ref{lem:residual-matchings}, $E(R^*)$ can be partitioned into
at most $3D$ globally rainbow matchings. When $D=0$, this family is
empty. Apply Lemma~\ref{lem:forest-extension} with $\ell=1$ separately
to each non-empty matching. Since $e(R^*)\le nD/2$ and
$\delta(G)\ge\alpha n$, all residual edges are covered by at most
\begin{equation}
\begin{aligned}
 &(24+o(1))\frac{ne(R^*)}{\delta(G)^2}
 +(2+o(1))\frac{3nD}{\delta(G)}\\
 &\hspace{25mm}\le
 (\frac{12}{\alpha^2}+\frac6\alpha+o(1))D
\end{aligned}
\label{eq:final-residual-cover}
\end{equation}
additional rainbow paths.

It remains to count the two families of paths. Since $k\le K$,
\eqref{eq:sequence-index-count}, \eqref{eq:exceptional-size}, and
\eqref{eq:exceptional-rate-explicit} give
\begin{equation}
\begin{aligned}
q
 &\le(1+\eta)\frac{n-|X|}{2}+20Ks+K
 \\&\le\frac n2+\frac{\eta n}{2}+20Kr_i n+K
 \\&\le\frac n2+\frac{\varepsilon n}{4}.
\end{aligned}
\label{eq:principal-path-count}
\end{equation}
For all sufficiently large $n$, the coefficient in
\eqref{eq:final-residual-cover} is at most $A_\alpha$. Therefore,
\eqref{eq:final-residual-degree} and the choices of $\lambda$ and $r_i$
bound the number of residual paths by
\begin{equation}
A_\alpha\bigl(\lambda+B_\alpha C_i r_i+o(1)\bigr)n
 \le\frac{\varepsilon n}{4}.
\label{eq:residual-path-count}
\end{equation}
Together with \eqref{eq:principal-path-count}, this proves
\begin{equation}
\rpc(G,c)
 \le\frac n2+\frac{\varepsilon n}{2}
 =(1+\varepsilon)\frac n2.
\end{equation}

For sharpness, take $G=K_n$. Give all edges of a fixed maximum matching
$M$ one common color and give every edge outside $M$ its own new color.
This is a proper edge-coloring, and a rainbow path contains at most one
edge of $M$. Hence
\begin{equation}
\rpc(K_n,c)\ge |M|=\lfloor\frac n2\rfloor.
\end{equation}
For every fixed $\alpha<1$, the graph $K_n$ satisfies
$\delta(K_n)\ge\alpha n$ for all sufficiently large $n$. Thus the
leading coefficient $1/2$ in part~(i) cannot be decreased.
\end{proof}

\section{Complete multipartite graphs}\label{sec:complete-multipartite}

Theorem~\ref{thm:linear-minimum-degree}(ii) proves the required upper
bound when the minimum degree is linear in $n$. It remains to consider
complete multipartite graphs in which one part contains almost all
vertices. We first prove a direct lemma for this case and then complete
the proof of Theorem~\ref{thm:multipartite-main}.

\subsection{A dominant part}\label{subsec:dominant-part}

\begin{lemma}[A dominant part]\label{lem:dominant}
Let $H$ be a graph on $h$ vertices, and let $B$ be an independent set
of $L$ new vertices. Let $G$ be obtained from $H$ by adding the
vertices of $B$ and all edges between $B$ and $V(H)$. If $L\geq20h$,
then every proper edge-coloring $c$ of $G$ satisfies
\begin{equation}
 \rpc(G,c)\leq\lceil\frac L2\rceil+3h.
\end{equation}
\end{lemma}

\begin{proof}
Fix a proper edge-coloring $c$ of $G$. The case $h=0$ is immediate,
so assume that $h\geq1$ and write
$V(H)=\{a_1,\ldots,a_h\}$. We first cover all edges between $B$ and
$V(H)$. We then cover the edges inside $H$.

Suppose first that $L=2p$. Each vertex $a_i$ is incident with $2p$
edges joining it to $B$. Since a path can use at most two of these
edges, we seek $p$ paths, each passing through
$a_1,\ldots,a_h$ in this order. Partition $B$ into two sets $X_0$ and
$X_1$ of size $p$. We alternate between these sets along each path, so
each path has the form
\begin{equation}
 b_0a_1b_1a_2\cdots a_hb_h,
 \qquad b_j\in X_{j\bmod 2}.
\end{equation}
For a fixed $i$, the vertex immediately before $a_i$ belongs to
$X_{(i-1)\bmod 2}$, while the vertex immediately after $a_i$ belongs
to $X_{i\bmod 2}$. Therefore, if each vertex of these two sets occurs
exactly once in the corresponding position among the $p$ paths, then
the paths use every edge between $a_i$ and $B$ exactly once.

To guarantee this property simultaneously for every $i$, we choose,
for each $j\in\{0,\ldots,h\}$, a bijection
\begin{equation}
 \pi_j:[p]\longrightarrow X_{j\bmod 2}.
\end{equation}
The path indexed by $t\in[p]$ is then
\begin{equation}\label{eq:dominant-paths}
 P_t=\pi_0(t)a_1\pi_1(t)a_2\cdots a_h\pi_h(t).
\end{equation}
The bijectivity of the maps $\pi_j$ guarantees that the paths
$P_1,\ldots,P_p$ cover all edges between $B$ and $V(H)$. It remains
to choose these bijections so that every $P_t$ is simple and rainbow.

Choose $\pi_0$ arbitrarily. Suppose that
$\pi_0,\ldots,\pi_{j-1}$ have been chosen for some $j\in[h]$. For
each $t\in[p]$, the part of $P_t$ already constructed is
\begin{equation}
 \pi_0(t)a_1\pi_1(t)a_2\cdots\pi_{j-1}(t)a_j.
\end{equation}
We must assign one vertex of $X_{j\bmod 2}$ to each of these partial
paths. The assignment must use every vertex of $X_{j\bmod 2}$
exactly once, must not repeat a vertex on any path, and must not
repeat a color.

Define an auxiliary bipartite graph with left side $[p]$ and right
side $X_{j\bmod 2}$. Suppose first that $j<h$. Join
$t\in[p]$ to $b\in X_{j\bmod 2}$ when $b$ does not occur on the
partial path indexed by $t$ and neither $c(a_jb)$ nor
$c(ba_{j+1})$ occurs on that path. Such a pair allows us to extend
the path through the segment $a_jba_{j+1}$. The two new edges have
different colors because they are incident with the same vertex
$b$ and the coloring is proper. When $j=h$, join $t$ to $b$ when
$b$ is new and $c(a_hb)$ does not occur on the partial path.

We show that this auxiliary graph has a perfect matching. First fix
$t\in[p]$. Fewer than $h$ vertices of $X_{j\bmod 2}$ are forbidden
because they already occur on the partial path. The partial path has
fewer than $2h$ colors. For each old color, properness implies that
at most one candidate $b$ gives this color on an edge from $a_j$,
and at most one candidate gives this color on an edge to
$a_{j+1}$. Thus the old colors forbid fewer than $4h$ candidates.
In total, fewer than $5h$ candidates are forbidden for $t$.

Now fix $b\in X_{j\bmod 2}$. Since each earlier map $\pi_k$ is a
bijection, fewer than $h$ partial paths already contain $b$. There
are at most two colors on the proposed new edges. Fix one of these
colors and one earlier edge position in the partial paths. As
$t$ ranges over $[p]$, the edges in this position form a star: their
endpoint in $H$ is fixed, while their endpoint in $B$ runs through
one of $X_0,X_1$. Since the coloring is proper, these edges have
distinct colors. Therefore the fixed new color occurs on at most
one partial path at each earlier edge position. There are fewer than
$2h$ earlier positions, so the proposed new colors forbid fewer
than $4h$ partial paths. Hence $b$ is forbidden for fewer than $5h$
values of $t$.

Since $p\geq10h$, every vertex of the auxiliary bipartite graph has
degree at least $p-5h\geq p/2$. This implies Hall's condition. Indeed,
let $S\subseteq[p]$. If $|S|\leq p/2$, then
$|N(S)|\geq p/2\geq|S|$. If $|S|>p/2$ and $|N(S)|<|S|$, choose a
vertex $b\in X_{j\bmod 2}\setminus N(S)$. Every neighbor of $b$
then lies in $[p]\setminus S$, which has size less than $p/2$. This
contradicts the minimum-degree bound.

Hall's theorem gives a perfect matching. If $t$ is matched to $b$,
set $\pi_j(t)=b$. Since the matching uses every vertex on each side
exactly once, $\pi_j$ is a bijection. By the definition of the
auxiliary graph, every extended path remains simple and rainbow.
Repeating this argument for $j=1,\ldots,h$ produces the $p$ paths in
\eqref{eq:dominant-paths}. As explained above, these paths cover all
edges between $B$ and $V(H)$.

Suppose next that $L=2p+1$. Remove one vertex $z\in B$. Since $L$ is
odd and $20h$ is even, the assumption $L\geq20h$ implies
$L-1\geq20h$. The even construction therefore covers all edges
between $B\setminus\{z\}$ and $V(H)$ with $(L-1)/2$ rainbow paths.

The $h$ edges joining $z$ to $V(H)$ have distinct colors because
they are all incident with $z$. Pair these edges at $z$. Each pair
forms a rainbow path of length two. If $h$ is odd, the remaining
edge forms a path of length one. Thus all edges incident with $z$
are covered by $\lceil h/2\rceil$ further rainbow paths.

It remains to cover $E(H)$. If $E(H)$ is empty, no further paths are
needed. Otherwise, let $\Gamma$ be the conflict graph with vertex set
$E(H)$, in which two vertices are adjacent if the corresponding edges
of $H$ share an endpoint or have the same color. Thus an independent
set in $\Gamma$ corresponds to a rainbow matching in $H$.

Consider an edge $xy\in E(H)$ of color $\gamma$. It conflicts with at
most $d_H(x)-1$ other edges containing $x$, at most $d_H(y)-1$ other
edges containing $y$, and at most $|E_\gamma\cap E(H)|-1$ other edges
of color $\gamma$. Therefore
\begin{equation}
 \Delta(\Gamma)\leq 2\Delta(H)+\mu_c(H)-3.
\end{equation}
By the standard greedy coloring bound
$\chi(\Gamma)\leq\Delta(\Gamma)+1$, the graph $\Gamma$ has a proper
vertex-coloring with at most
$2\Delta(H)+\mu_c(H)-2$ colors. Its color classes are independent
sets and hence give a partition of $E(H)$ into that many rainbow
matchings. Since $\Delta(H)\leq h-1$ and every color class in $H$ is a
matching, $\mu_c(H)\leq\lfloor h/2\rfloor$. Consequently, at most
$2h+\lfloor h/2\rfloor-4$, and hence at most
$2h+\lfloor h/2\rfloor$, rainbow matchings are needed.

Consider one such matching,
$\{x_1y_1,\ldots,x_ty_t\}$. We turn it into one rainbow path by
joining consecutive matching edges through distinct vertices
$b_1,\ldots,b_{t-1}\in B$. The intended path is
\begin{equation}
 x_1y_1b_1x_2y_2b_2\cdots b_{t-1}x_ty_t.
\end{equation}
Choose the vertices $b_i$ successively. When choosing $b_i$, fewer
than $3t$ colors have already been used: the $t$ colors of the
matching edges and the colors of the earlier connector edges. By
properness, each used color excludes at most one candidate through
$y_i$ and at most one candidate through $x_{i+1}$. Thus colors
exclude fewer than $6t$ candidates. The earlier connector vertices
exclude fewer than $t$ further candidates. Since the matching has
$t\leq\lfloor h/2\rfloor$ edges, fewer than
$7t\leq7h/2<L$ candidates are excluded. Hence a suitable $b_i$
exists at each step. Properness also ensures that the two new edges
incident with $b_i$ have different colors. The resulting path is
therefore simple and rainbow.

The connector edges may already occur in the paths covering the
edges between $B$ and $V(H)$. This is allowed because a path cover
is not required to be edge-disjoint.

If $L$ is even, the total number of paths is at most
\begin{equation}
 \frac L2+2h+\lfloor\frac h2\rfloor
 \leq\frac L2+3h.
\end{equation}
If $L$ is odd, the total number is at most
\begin{equation}
 \frac{L-1}{2}+\lceil\frac h2\rceil
 +2h+\lfloor\frac h2\rfloor
 \leq\lceil\frac L2\rceil+3h.
\end{equation}
This proves the lemma.
\end{proof}

\subsection{Proof of Theorem~\ref{thm:multipartite-main}}\label{subsec:multipartite-final}

\begin{proof}
It is enough to consider $0<\varepsilon\leq1$. Choose
\begin{equation}\label{eq:alpha-choice}
 0<\alpha<\min\{\frac1{22},\frac\varepsilon{100}\}.
\end{equation}
Let $A$ be a largest part in $G$, put $M=|A|$, and write
$S=n-M=\delta(G)$. If $S\geq\alpha n$, the result follows from
Theorem~\ref{thm:linear-minimum-degree}(ii).

Suppose that $S<\alpha n$. Then $M\geq20S$ by
\eqref{eq:alpha-choice}. The graph induced by $V(G)\setminus A$ is a
graph $H$ on $S$ vertices. The part $A$ is independent, and every
vertex of $A$ is adjacent to every vertex of $H$. Therefore
Lemma~\ref{lem:dominant} gives
\begin{equation}
 \rpc(G,c)\leq\lceil\frac M2\rceil+3S.
\end{equation}
Since $G$ has at least two parts, its smallest part has size at most
$S$. Hence $\Delta(G)=n-s\geq M$, and therefore
$Q(G)\geq\Delta(G)/2\geq M/2$. Thus
\begin{equation}
 \frac{\lceil M/2\rceil+3S}{Q(G)}
 \leq1+\frac{6S+1}{M}
 \leq1+\frac{6\alpha}{1-\alpha}
       +\frac1{(1-\alpha)n}
 <1+\varepsilon
\end{equation}
for all sufficiently large $n$. This proves \eqref{eq:multipartite-upper}
with a threshold depending only on $\varepsilon$. Taking the maximum
over proper colorings and using \eqref{eq:universal-lower} gives
\eqref{eq:multipartite-asymptotic}. All thresholds depend only on
$\varepsilon$, so the conclusion is uniform in the number and sizes of
the parts.
\end{proof}

\bibliographystyle{plain}
\addcontentsline{toc}{section}{Bibliography}
\bibliography{rainbow}
\end{document}